\documentclass[11pt]{article}
\usepackage{amssymb,latexsym,amsmath}
\usepackage{amsthm}
\usepackage{epsfig}
\usepackage{comment}
\usepackage{accents}
\usepackage{color}

\def\BB4#1{\textcolor{red}{#1}}

\def\rank{{\rm rank \, }}

\renewcommand{\deg}{{\rm deg} \,}

\newcommand{\mo}[1]{\left|{#1}\right|}

\newcommand{\conj}[1]{\left\{ {#1}\right\}}

\numberwithin{equation}{section}

\newtheorem{theorem}{Theorem}[section]
\newtheorem{lemma}[theorem]{Lemma}
\newtheorem{corollary}[theorem]{Corollary}

\newtheorem{conjecture}{Conjecture}
\newtheorem{Definition}[theorem]{Definition}

\newtheorem{Remark}[theorem]{Remark}
\newenvironment{remark}{\begin{Remark}\rm}{\end{Remark}}
\newtheorem{Example}[theorem]{Example}
\newenvironment{example}{\begin{Example}\rm}{\end{Example}}
\newtheorem{Problem}[theorem]{Problem}

\begin{document}
\title{Algebraic properties of robust Pad\'e approximants}
\author{Bernhard Beckermann,
Ana C. Matos \thanks{
Laboratoire Painlev\'e UMR 8524, UFR Math\'ematiques --
M3, Universit\'e de Lille, F-59655 Villeneuve d'Ascq CEDEX, France. E-mail:
{\tt $\{$bbecker,matos$\}$@math.univ-lille1.fr}. Supported in part by the Labex CEMPI  (ANR-11-LABX-0007-01).}}

\date{}
\maketitle
\centerline{\em Dedicated to the memory of our friend Herbert Stahl and our colleague A.A.\ Gonchar.}

\begin{abstract}
For a recent new numerical method for computing so-called robust Pad\'e approximants through SVD techniques, the authors gave numerical evidence that such approximants are insensitive to perturbations in the data, and do not have so-called spurious poles, that is, poles with a close-by zero or poles with small residuals. A black box procedure for eliminating spurious poles would have a major impact on the convergence theory of Pad\'e approximants since it is known that convergence in capacity plus absence of poles in some domain $D$ implies locally uniform convergence in $D$.

  In the present paper we provide a proof for forward stability (or robustness), and show absence of spurious poles for the subclass of so-called well-conditioned Pad\'e approximants. We also
  give a numerical example of some robust Pad\'e approximant which has spurious poles, and
  discuss related questions. It turns out that it is not sufficient to discuss only linear algebra properties of the underlying rectangular Toeplitz matrix, since in our results other matrices like Sylvester matrices also occur. These types of matrices have been used before in numerical greatest common divisor computations.

\end{abstract}


\vspace{0.5cm}

\noindent {\bf AMS Classification (2010): } 41A21, 65F22

\section{Introduction and statement of the main results}\label{sec1}

A popular method for approximation, for analytic continuation or for detection of singularities of a function $f$ knowing 
the first terms of its Taylor expansion at zero $f(z)= \sum_{j=0}^{m+n} c_j z^j+\mathcal O(z^{m+n+1})_{z\to 0}$ is to compute its $[m|n]$ Pad\'e approximant at zero $p/q$, namely a rational function satisfying
\begin{equation} \label{definition_Pade}
     p(z)=\sum_{j=0}^m p_j z^j , \quad
     q(z)=\sum_{j=0}^n q_j z^j \not\equiv 0 , \quad
     f(z)q(z)-p(z) = \mathcal O(z^{m+n+1})_{z\to 0}.
\end{equation}
It is well known \cite[Section~1]{BGM} that there always exists an $[m|n]$ Pad\'e approximant: we just have to find a non-trivial solution of the homogeneous system of $n$ equations and $n+1$ unknowns with Toeplitz structure
\begin{equation} \label{homogeneous_system}
    C \, \mbox{vec}(q)=0, \quad C = \left[\begin{array}{cccc}
c_{m+1} & \cdots & c_{m-n+2} & c_{m-n+1}\\
c_{m+2} & \cdots & c_{m-n+3} & c_{m-n+2}\\
\vdots &  & \vdots&\vdots\\
c_{m+n} & \cdots & c_{m+1} & c_{m}
\end{array}\right]
, \quad
\mbox{vec}(q)=\left[\begin{array}{cc} q_0 \\ q_1 \\ \vdots \\ q_n
\end{array}\right],
\end{equation}
with the convention $c_j=0$ for $j<0$,
and then find the coefficients of $p$ from \eqref{definition_Pade}.
Whereas \eqref{homogeneous_system} has infinitely many solutions, it is also known \cite{BGM} that the rational function $p/q$ is unique.

Though many theoretical results \cite[Section~6]{BGM} show the usefulness of sequences of Pad\'e approximants in approximating $f$ or its singularities, there are drawbacks making it somehow difficult to interpret correctly the approximation power of such approximants: it might happen that the rational function has poles at places where the function $f$ has no singularities, so-called spurious poles. This somehow vague notion needs some more explanation, for a precise (asymptotic) definition see the work \cite[Definition~8]{St97a} or \cite{Stahl} of Stahl: the Pad\'e convergence theory like the Nuttall-Pommerenke Theorem for meromorphic functions $f$ \cite[Theorem~6.5.4]{BGM} or the celebrated Stahl Theorem for algebraic functions $f$ \cite[Theorem~1.2]{St97b}, \cite[Theorem~6.6.9]{BGM} (or more general multivalued functions) tells us that there are domains $D$ of analyticity of $f$ such that the $[n|n]$ Pad\'e approximants tend for $n \to \infty$ to $f$ in capacity on any compact subset $K$ of $D$. That is,
given any threshold $\epsilon>0$, the set of exceptional points in $K$ where the
error is larger than $\epsilon$ becomes quickly ``small'', 
see, e.g., \cite[Section~6.6]{BGM} and the references therein. By the Gonchar Lemma \cite[Lemma~1]{G75}, convergence in capacity and absence of poles implies uniform convergence, but there are examples showing that there might be poles of an infinite subsequence of $[n|n]$ Pad\'e approximants in $K$, which of course makes it impossible to have uniform convergence in $K$. Stahl shows in \cite[Theorem~3.7]{St96} that one can establish uniform convergence for the special case of hyperelliptic functions by simply dropping all terms in a partial fraction decomposition with poles in $D$. More generally, for algebraic functions, Stahl mentions in \cite[Remark~(8) for Theorem~1.2]{St97b} an elimination procedure for spurious poles, but without giving details.

The notion \cite{Stahl} of asymptotically spurious poles of course is intractable on a computer since we are able to compute only finitely many approximants. In addition, the computed approximants will be affected by finite precision arithmetic, or by noise on the given Taylor coefficients. It was suggested by Froissart \cite{Froissart} and further analyzed for particular functions in \cite{Bessis,GP97,GP99} that instead we should identify poles of Pad\'e approximants which come along with a ``close-by''
zero, so-called {\em Froissart doublets}. The occurrence of such doublets is observed experimentally to increase in case of noise on the Taylor coefficients \cite{Bessis}. Stahl shows in \cite{Stahl} that in fact asymptotically spurious poles give raise to ``asymptotical Froissart doublets''.

Another popular method to detect ``doubtful'' poles, adapted for instance in \cite{GGT12}, is to identify poles $z_k$ which have ``small'' residuals $a_k$ corresponding to terms $\frac{a_k}{z-z_k}$ in the partial fraction decomposition of a Pad\'e approximant $p/q$. 
Notice that such poles are generically of multiplicity one.

Before going further, some notation. We denote by $\mathcal R_{m,n}$ the set of rational functions with numerator (and denominator) degree not exceeding $m$ (and $n$, respectively). In what follows, $\| \cdot \|$ will always denote Euclidian norm together with the induced spectral norm
of a possibly rectangular matrix $A$. The matrices $A$ under consideration will always have full row rank $\ell$, in which case we may write the spectral condition number as
$$
     \kappa(A)=\frac{\sigma_1(A)}{\sigma_\ell(A)} = \| A \| \, \| A^\dagger \|
$$
with $\sigma_j(A)$ the $j$th largest singular value of $A$, and the pseudoinverse $A^\dagger=A^* (A A^*)^{-1}$.
We notice that a change of norms might improve some of our estimates below, in particular we do not claim that any of the powers of $m+n+1$ occurring below are optimal. Hence we will sometimes use the writing $a_1 \lesssim a_2$ meaning that there exist modest constants $b,r>0$ 
not depending on $f$ or $m,n$ such that $a_1 \leq b (m+n+1)^r \, a_2$. Also, $a_1\sim a_2$ means that $a_1 \lesssim a_2$  and $a_2 \lesssim a_1$.
As suggested in \cite{GGT12}, before computing Pad\'e approximants of $f$ one should replace $f$ by
a suitably scaled counterpart $a f(bz)$ with nonzero scalars $a,b$ chosen such that
\begin{equation} \label{scaling}
    \sum_{j=0}^{m+n} |c_j|^2=1.
\end{equation}
Here the rescaling factor $b$ should be chosen in order to obtain quantities $|c_j|\leq 1$ of comparable size, which asymptotically means that we rescale the complex plane in a way such that a meromorphic function $f$ becomes analytic in $|z|<1$. Finally, in order to simplify notation, in what follows we always fix $m$ and $n$ and drop these indices.

\subsection{Robust Pad\'e approximants, degeneracy and related matrices}
Recently \cite{GGT12}, Gonnet, G\"uttel and Trefethen suggested
the interesting concept of a robust $[m'|n']$ Pad\'e approximant $p/q$ based on SVD computations. This object essentially is an $[m|n]$ Pad\'e approximant (at least for exact arithmetic) for suitably chosen $m \leq m'$ and $n\leq n'$. Though the suggested numerical method to find $m,n$ from $m',n'$ is much more elaborate, one may get an idea of the method by thinking of $(m,n)$ as being the upper left corner of a ``numerical block'' of the Pad\'e table containing the coordinate $(m',n')$, or being on the upper or left border of such a ``block'' and on the same diagonal $m'-n'=m-n$. In the numerical
experiments reported in \cite{GGT12}, the shape of such a ``numerical block'' is either a (finite or infinite) square or an infinite diagonal. Their robust $[m|n]$ Pad\'e approximant $p/q$ has the following properties
\begin{description}
\item[(P1)]  it is {\em nondegenerate} in the sense that the polynomials $p$ and $q$ are co-prime, and that the {\em defect} $\min \{ m-\deg p,n-\deg q \}$ is equal to zero;
\item[(P2)]  the $n$th largest singular value $\sigma_n(C)$ is larger than a certain threshold;
\item[(P3)]  the denominator is given by choosing as $\mbox{vec}(q)$ a right singular vector of norm $1$ corresponding to the singular value  $\sigma_{n+1}(C)=0$.
\end{description}
We can read from {\bf (P2),(P3)} that indeed $C$ has maximal 
numerical rank $n$, and thus $\mbox{vec}(q)$ spans the numerical kernel of $C$, see also \cite{AI}.
Moreover, according to \eqref{scaling} and {\bf (P2)}, the 
condition number $\kappa(C)$ 
will be of moderate size.

The authors in \cite{GGT12} use analogies from well-known regularization techniques for linear algebra problems in order to justify theoretically their approach.
Their paper contains many numerical examples
which lead one to believe that these new ``regularized'' approximants are indeed robust, that is, small perturbations in the input like noisy Taylor coefficients produce similar approximants, see also \S\ref{sec_robust} below for this notion of robustness or forward stability. Also, in all numerical
 experiments reported in \cite{GGT12}, these robust approximants do no longer have Froissart doublets nor small residuals.
The aim of the present paper is to give some theoretical results complementing these numerically observed phenomena.
For instance, we present a numerical example of robust approximants where spurious poles have not been eliminated. In addition, we describe a subclass of robust approximants where we can insure that we have eliminated spurious poles.
All our statements only apply to nondegenerate $[m|n]$ Pad\'e approximants $p/q$, and we will see that {\bf (P2)} will enable us to show that the underlying nonlinear map is forward well-conditioned. 
For the backward condition number, for Froissart doublets or for small residuals, other matrices $T, S,$ and $Q$ do occur, which are defined as follows:

We first observe that \eqref{definition_Pade}, \eqref{homogeneous_system} is equivalent to solving
\begin{equation} \label{homogeneous_system2}
    T \, \left[\begin{array}{cc}
\mbox{vec}(p) \\ \mbox{vec}(q)\end{array}\right]=0, \quad T = \left[\begin{array}{cccccccc}
1 & 0 & \cdots & 0 & -c_0 & 0 & \cdots & 0 \\
0 & 1 & \ddots & \vdots & -c_1 & -c_0 & \ddots & \vdots \\
\vdots & \ddots & \ddots & 0 & \vdots & \ddots & \ddots & 0 \\
0 & \cdots & 0 & 1 & -c_m & & \ddots & -c_0 \\
0 & \cdots & \cdots & 0 & -c_{m+1} & \cdots & \cdots & -c_1 \\
\vdots &  & & \vdots&\vdots &  & & \vdots\\
0 & \cdots & \cdots & 0 &
-c_{m+n} & \cdots & \cdots & -c_{m}
\end{array}\right]\in\mathbb C^{(m+n+1)\times (m+n+2)},
\end{equation}
$T$ being block upper triangular, with the lower right block given by $-C$. We will also require the two matrices
\begin{equation} \label{def_Sylvester}
   Q = \left[\begin{array}{ccccccc} q_0 & 0 & \cdots & \cdots & \cdots & 0 \\
   \vdots & \ddots & \ddots & & & \vdots \\
   q_n & & q_0 & 0 & \cdots & 0 \\
   0 & \ddots & & \ddots & \ddots & \vdots\\
   \vdots &\ddots & \ddots & & \ddots & 0\\
   0 & \cdots & 0 & q_n & \cdots & q_0
\end{array}\right], \,\,
  S=
  \left[\begin{array}{cccccccc}
  q_0 & 0 & \cdots & 0 &  -p_0 & 0 & \cdots & 0 \\
   \vdots & \ddots & \ddots & \vdots & \vdots & \ddots & \ddots & \vdots \\
   q_n & & \ddots & 0 & -p_m & & \ddots & 0 \\
   0 & \ddots & & q_0 &0 & \ddots & & -p_0  \\
   \vdots &\ddots & \ddots & \vdots & \vdots &\ddots & \ddots & \vdots\\
   0 & \cdots & 0 & q_n & 0 & \cdots & 0 & -p_m
   \end{array}\right],
\end{equation}
with $Q \in \mathbb C^{(m+n+1)\times (m+n+1)}$, and $S=S(q,-p)\in\mathbb C^{(m+n+1)\times (m+n+2)}$ having one more row and two more columns than the usual Sylvester matrix of two polynomials. Notice that these matrices are related through
\begin{equation} \label{S=QT}
    S = Q T .
\end{equation}

\subsection{Continuity and conditioning of the Pad\'e map}\label{sec_robust}
For defining a (nonlinear) Pad\'e map $F$
\begin{equation} \label{def_Pade_map1}
      F : \mathbb C^{m+n+1} \ni c = (c_0,...,c_{m+n})^t \mapsto
      y=\left[\begin{array}{cc} \mbox{vec}(p) \\ \mbox{vec}(q)
       \end{array}\right] \in \mathbb C^{m+n+2}
\end{equation}
mapping the vector of $(m+n+1)$ Taylor coefficients to the coefficient vector in the basis of monomials of the numerator and denominator of an $[m|n]$ Pad\'e approximant $p/q$ we have to be a bit careful due to degeneracies in the Pad\'e table, also we have to fix the normalization (norm and phase) of the coefficients. Uniqueness is obtained by taking any $p,q$ of degree at most $m$, and $n$, respectively, satisfying \eqref{homogeneous_system2}, by canceling out a possible non-trivial greatest common divisor such that $q(0)\neq 0$ (since $p(0)=c_0 q(0)$), and then normalize in a suitable manner by a complex scalar, here
\begin{equation} \label{def_Pade_map2}
      \| F(c) \|^2 = \| \mbox{vec}(p) \|^2 + \| \mbox{vec}(q) \|^2 = 1 , \quad q(0)>0.
\end{equation}
Notice that a non-trivial greatest common divisor only occurs for degenerate Pad\'e approximants, and only here it might happen that $T F(c)\neq 0$. Also, $F$ is neither injective nor surjective.
By adapting the techniques of \cite{WW}, one may show the following result which is stated here without proof and which shows the importance of degeneracy.

\begin{theorem}\label{thm_continuity}
   $F$ is continuous in a neighborhood of $c$ if and only if its $[m|n]$ Pad\'e approximant $F(c)$ is nondegenerate.
\end{theorem}

For studying 
conditioning we will restrict ourselves to the {\em real Pad\'e map}, namely the restriction of $F$ onto $\mathbb R^{m+n+1}$, also denoted by $F$, and hence $F(c)\in \mathbb R^{m+n+2}$. For the convenience of the reader, let us recall two different concepts of condition numbers measuring both the worst case amplification of infinitesimally small relative errors: for the forward conditioning $\kappa_{for}(F)(c)$ one is interested whether small errors $\widetilde c -c$ in the data gives an answer $F(\widetilde c)$ close to $F(c)$.
In contrast, for the backward conditioning one considers $\widetilde y$ close to $F(c)$ and asks whether $\widetilde y$ is the right answer $F(\widetilde c)$ for some $\widetilde c$ close to $c$. However, due to the lack of surjectivity, it could be necessary to project first the perturbed value $\widetilde y$ on the image of $F$, and we might need additional assumptions in order to insure that the value $\mbox{dist}(\widetilde y, F(\mathbb R^{m+n+1}))$ is attained at some $F(\widetilde c)$. Also, in general there might be several such arguments $\widetilde c$ due to the lack of injectivity and we have to find the one closest to $c$.

However, as we see in Theorem~\ref{thm_stability}(a),(b) below, for the real Pad\'e map the situation is much less involved: for instance, we show that $F$ is injective in a neighborhood of a point of continuity. Also, since $\| c\| = \| F(c) \|=1$ by \eqref{scaling} and \eqref{def_Pade_map2}, we may replace relative errors by absolute errors in the definition of conditioning, which make our formulas more readable.


\begin{theorem}\label{thm_stability}
   Suppose that $F$ is continuous in a neighborhood of $c\in \mathbb R^{m+n+1}$, that \eqref{scaling} holds, and that the matrix $T$  of (\ref{homogeneous_system2}) is defined by $c$ and $Q$  of (\ref{def_Sylvester}) by $F(c)$. Then the following statements hold.
   \begin{description}
   \item[(a)]
      There exists $\mathcal U \subset \mathcal R^{m+n+1}$, a neighborhood of $c$, and $\mathcal V \subset \mathbb S^{m+n+2}:=\{ y \in \mathcal R^{m+n+2}: \| y \| = 1 \}$, a relative neighborhood of $F(c)$ on the unit sphere $\mathbb S^{m+n+2}$ such that the restriction $F:\mathcal U \to \mathcal V$ is a diffeomorphism, and we have the Jacobian  $J_{F}(c) = T^\dagger Q$.
   \item[(b)]
      For any $\widetilde y$ sufficiently close to $F(c)$, the projection of $\widetilde y$ onto $F(\mathbb R^{m+n+1})$ exists, and is given by $\widetilde y/\| \widetilde y \|\in \mathcal V$.
   \item[(c)]
      The forward condition number is given by
       \begin{equation} \label{stability_forward}
          \kappa_{for}(F)(c):=\limsup_{\widetilde c \to c}
          \frac{\| F(\widetilde c) - F(c) \|}{\| \widetilde c - c \|}
          = \| T^\dagger Q \|.
       \end{equation}
   \item[(d)]
      The backward condition number is given by
   \begin{equation} \label{stability_backward}
       \kappa_{back}(F)(c):=\limsup_{\widetilde y \to F(c)}
       \frac{\inf \{ \| \widetilde c  - c\| : F(\widetilde c)=\widetilde y/\| \widetilde y\|\} }
       {\| \widetilde y - F(c) \|}
       = \| J_F(c)^\dagger \| =  \| Q^{-1} T \|.
   \end{equation}
   \end{description}
\end{theorem}

We know from \eqref{scaling} and \eqref{def_Pade_map2} (see also Lemma~\ref{lem_T_C} below) that both matrices $T$ and $Q$ have a norm not larger than $\sqrt{m+n+2}$. Thus we learn from Theorem~\ref{thm_stability}(c) that the real Pad\'e map is forward (backward)  well-conditioned at $c$ provided that the smallest singular value of $T$ (and of $Q$, respectively) is not too small.
It is shown in Lemma~\ref{lem_T_C} below that the smallest singular values of $T$ and $C$ are of the same magnitude. Thus condition {\bf (P2)} insures that the real Pad\'e map is forward well-conditioned.

In our proof of Theorem~\ref{thm_stability}(c) we exploit a well-known formula for $\kappa_{for}(F)(c)$ in terms of the Jacobian of $F$. To our knowledge, similar formulas for $\kappa_{back}(F)(c)$ in terms of the pseudoinverse of the Jacobian have not been established before in the literature. The occurrence of a sub-matrix of $Q$ in the backward conditioning of the Pad\'e denominator map has been noticed before by S.\ G\"uttel (personal communication).

\subsection{Well-conditioned rational functions and spurious poles}
Let us now turn to the subject of spurious poles, which in the present paper we study for general rational functions and not only for Pad\'e approximants. It will be shown in Lemma~\ref{lem_degeneracy} below that $p/q$ is nondegenerate if and only if the corresponding matrix $S$ has full row rank. In a numerical setting, rank deficiency is typically excluded in requiring a condition number of modest size. In what follows, we will refer to rational functions as well-conditioned if the corresponding matrix $S$ has a modest condition number. As we show in the next theorem, for well-conditioned rational functions we are able to control the occurrence both of Froissart doublets and of small residuals.
We refer to \cite{BeLa} and Lemma~\ref{ngcd} below for other known results on Froissart doublets but, to our knowledge, no such result has been published before for residuals.

In the statement 
below we will make use of the uniform chordal metric in the set $\mathcal M_K$ of functions meromorphic in some compact $K\subset \mathbb C$ being defined by
\begin{equation} \label{chordal}
     \chi_K(r,\widetilde r) = \max_{z\in K} \chi(r(z),\widetilde r(z)) , \quad
     \chi(a,b)=\frac{|a-b|}{\sqrt{1+|a|^2}\sqrt{1+|b|^2}} .
\end{equation}
Such a metric is 
useful to study questions of uniform convergence for rational or meromorphic functions since such functions are continuous in $K$ with respect to the chordal metric. A different 
uniform metric has been also employed in \cite{WW} for measuring the distance of two rational functions for the continuity of the Pad\'e map. We will discuss the link with the distance of two coefficient vectors in more detail in \S\ref{sec3}. Notice that the next statement does not only cover Froissart doublets and small residuals of $r=p/q$ but also of rational functions $\widetilde r=\widetilde p/\widetilde q$ close to $r$, as those constructed in \cite{GGT12} where small leading coefficients in $p$ or $q$ are replaced by $0$.

\begin{theorem}\label{thm_rational}
   Let the two polynomials $p$ of degree $\leq m$ and $q$ of degree $\leq n$ be such that $r=p/q$ is nondegenerate. Then the following statements hold for the matrix $S=S(q,-p)$.
   \begin{description}
   \item[(a)] For any meromorphic function $\widetilde r\in \mathcal M_{\mathbb D}$ with $\chi_{\mathbb D}(r,\widetilde r)\leq 1/3$, the
       Euclidian
       distance of any
       pair
       of zeros and poles of $\widetilde r$ in the unit disk is bounded below by $1/(3 \sqrt{2} (m+n+1)^{3/2} \, \kappa(S))$.
   \item[(b)] For any rational function $\widetilde r\in \mathcal R_{m,n}$ with $2 \, (m+n+1)^{2} \kappa(S)^2 \chi_{\mathbb D}(r,\widetilde r)\leq 1/3$, the modulus of any residual of a simple pole in the unit disk of $\widetilde r$ is bounded below by $1/((2(m+n+1))^{3/2} \kappa(S))$.
   \end{description}
\end{theorem}

   Numerical results presented in Example~\ref{exa3} below indicate that both lower bounds of Theorem~\ref{thm_rational} can be approximately attained.
It seems for us that, due to the use of the basis of monomials, the occurrence of the unit disk $\mathbb D$ in Theorem~\ref{thm_rational} is natural. For the case $m=n$ of diagonal rational functions $r,\widetilde r \in \mathcal R_{n,n}$, we could also obtain results outside of the unit disk, by considering the reversed numerator and denominator polynomials (for which $\kappa(S)$ remains unchanged).

Let us finally turn to convergence questions for robust Pad\'e approximants. In \cite[\S 8]{GGT12}, Gonnet, G\"uttel and Trefethen asked whether there are analogues of classical convergence theorems by Stahl and Pommerenke for robust Pad\'e approximants where the absence of spurious poles would enable to obtain not only convergence in capacity but uniform convergence. To be more precise, the authors suggest to compute robust Pad\'e approximants of type $[m_k|n_k]$ for increasing sequences of numbers $m_k,n_k$, where each approximant is computed using a threshold $tol_k$ possibly tending to zero for $k\to \infty$.
Notice that a variable threshold does no longer allow a simple control of spurious poles through our Theorem~\ref{thm_rational}.
But quite often there are only a finite number of distinct robust Pad\'e approximants following for instance a diagonal path $m_k=n_k=k$ if one uses a fixed threshold for all approximants. For instance, the numerical experiments for the exponential function with $tol_k=10^{-14}$ as reported in \cite[Fig.\ 5.1]{GGT12} tell us that there are only $8$ distinct robust Pad\'e approximants on the diagonal, since all approximants of type $[n|n]$ for $n \geq 8$ reduce to the one for $n=7$.

This vague observation can be made more explicit for
  Stieltjes functions $f$, since here the matrix $C$ has a condition number which grows quickly with $n$, see \cite{beck} for results on the condition number of positive definite Hankel matrices. For general functions $f$, we have the following result.

\begin{theorem}\label{thm_convergence}
   Let $r=p/q\in \mathcal R_{m,n}$ be nondegenerate and
   $\widetilde r=\widetilde p/\widetilde q\in \mathcal R_{m-1,n-1}$.
   Then $2 \, \chi_{\mathbb D}(r,\widetilde r) \kappa(S)^2 \geq (m+n+1)^{-2}$ for the matrix $S=S(q,-p)$.
\end{theorem}

We feel that it should be possible to establish an improved version of Theorem~\ref{thm_convergence} where $\kappa(S)^2$ is replaced by a term of order $\kappa(S)$. Such a result is given in Corollary~\ref{cor_convergence} below at least for the special case where $r,\widetilde r$ are two succeeding Pad\'e approximants on a diagonal.
Notice also that Theorem~\ref{thm_convergence} implies for the rational function $\widetilde r$ of Theorem~\ref{thm_rational}(b) to be nondegenerate.

Roughly speaking, we learn from Theorem~\ref{thm_convergence} that for any function $f$ which can be well approximated by some element of  $\mathcal R_{m-1,n-1}$ with respect to the uniform chordal metric in the unit disk, its $[m|n]$ Pad\'e approximant $r$ either does not have a small approximation error $\chi_{\mathbb D}(f,r)$, or otherwise the number $\kappa(S)$ is necessarily ``large''.
Since we feel that on a computer it is preferable to compute only well-conditioned rational functions, this could lead to an early stopping criterion for computing only Pad\'e approximants of small order. Such a stopping criterion would however require a systematic study of the error of best rational approximants with respect to the uniform chordal metric, which to our knowledge is an open problem, beside the negative result \cite[Theorem~3.1]{Darras}. Another impact of Theorem~\ref{thm_convergence} could be to introduce in the computation of Pad\'e approximants a penalization term taking care of a modest $\kappa(S)$ or some more appropriate estimator, inspired by techniques from inverse problems. But this is far beyond the scope of the present paper.

The remainder of the paper is organized as follows.
  \S\ref{sec_num} contains some numerical
   experiments which confirm our theoretical findings.
In \S\ref{sec2} we give
auxiliary statements and provide a proof of Theorem~\ref{thm_stability} on the conditioning of the real Pad\'e map.
\S\ref{sec3} is devoted to the study of distances of rational functions, we will show in Theorem~\ref{thm_distance} that in some cases the uniform chordal metric is close to forming differences of scaled coefficient vectors. A proof of Theorem~\ref{thm_rational} and Theorem~\ref{thm_convergence} is provided in \S\ref{sec4}. In Section \S\ref{sec5} we report about some previous work on related fields like numerical GCDs, condition number estimators, and look-ahead procedures for computing Pad\'e approximants. 
A summary of our work and concluding remarks can be found in \S\ref{sec6}.

\section{Some numerical experiments}\label{sec_num}
In this section we present examples of subdiagonal Pad\'e approximants ($m=n-1$) for three functions, namely
\begin{equation} \label{exa}
      f_1(z)=\int_{-1}^1 \frac{1}{\sqrt{1-x^2}} \frac{dx}{1-xz} ,
      \quad
      f_2(z) = \exp(z) ,
      \quad
      f_3(z) = \sum_{j=0}^{2N} c_3(j) z^j , \quad c_3={\tt randn}(2N) ,
\end{equation}
the first one a Stieltjes
function analytic in $|z|<1$, and the second (and third) one an entire function with quickly decaying Taylor coefficients (and random coefficients, respectively).
For each $m+1=n=1,...,N$, we first normalize the vector of the first $m+n+1$ Taylor coefficients following \eqref{scaling} by dividing by the norm. Subsequently, we compute the denominator coefficients
using the SVD, the corresponding coefficients of the numerator by multiplying
by a submatrix of $T$, and then normalize following \eqref{def_Pade_map2} by dividing by the norm. It turns out that all 
subdiagonal approximants are nondegenerate,
though there are $2\times 2$ blocks in the Pad\'e table of the even function $f_1$.

We draw in Fig~\ref{fig1}, Fig~\ref{fig2}, and Fig~\ref{fig3} the
 condition number of the four matrices $C$, $T$, $S$ and $Q$, as well as the norm of the two matrices $T^\dagger Q$ and  $Q^{-1} T$ occurring in Theorem~\ref{thm_stability}(c),(d). One observes that always $\kappa(C)$ and $\kappa(T)$ are of the same magnitude, and that $\max \{ \kappa(Q),\kappa(T) \}\lesssim \kappa(S)$. 
 These properties are shown analytically in Lemma~\ref{lem_T_C} below.
 It is also not difficult to establish the inequalities $\| Q^{-1} T \|\lesssim \kappa(Q)$ and $\| T^{\dagger} Q \| \leq \kappa(T)$, but we also observe without proof in our numerical experiments that $\| Q^{-1} T \|\approx \kappa(Q)$ and $\| T^{\dagger} Q \|\approx \kappa(T)$, up to some artifacts for the exponential function and $n \geq 11$ in Fig~\ref{fig2} which we believe are due to rounding errors.

In order to discuss the sharpness of Theorem~\ref{thm_rational}, we also draw the reciprocal values of
\begin{eqnarray*}
   && Froissart = \min \{ | \sigma - \tau |: p(\sigma)=0, q(\tau)=0, |\tau|\leq 1 \},
   \\
   && Residual = \min \{ | \frac{p(\tau)}{q'(\tau)}| : q(\tau)=0, |\tau|\leq 1 \},
\end{eqnarray*}
in case where the $[n-1|n]$ Pad\'e approximant has at least one pole in the unit disk. Below we give some specific comments for each of the three functions.

\begin{figure}[htb]
    \centerline{\includegraphics[width=0.75\textwidth]{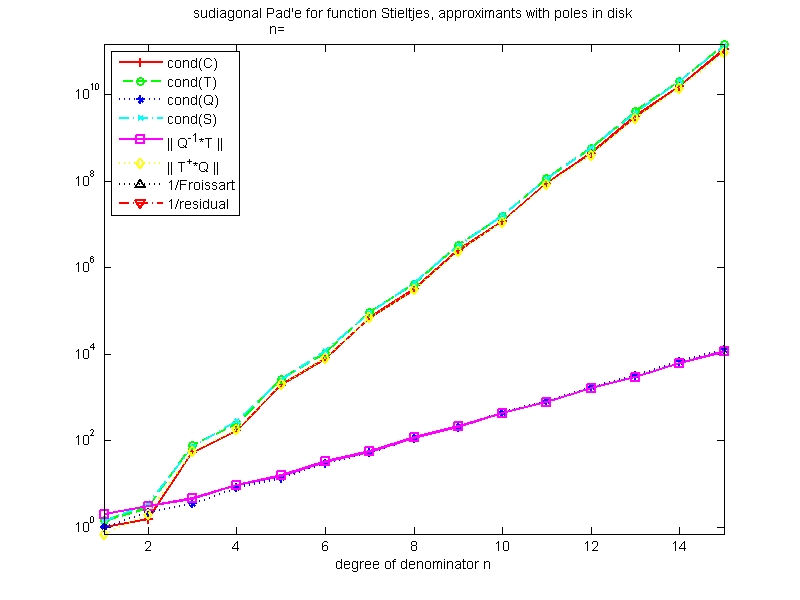}}
  \caption{Condition numbers related to the Stieltjes function $f_1$.}\label{fig1}
\end{figure}

\begin{example}\label{exa1}
   The $[n-1|n]$ Pad\'e approximant for $n=1,...,N=15$ of the Stieltjes function $f_1$ in \eqref{exa} does not have poles in the unit disk, even in presence of rounding errors.
   We observe from Fig.~\ref{fig1} that $\kappa(S)$ and $\kappa(T)$ have the same magnitude, and are growing exponentially in $n$. Also, $\kappa(Q)$ is growing exponentially in $n$, but less quickly.
   \\
   This example clearly shows that $\kappa(S)$ large does not imply the existence of a Froissart doublet or a small residual in the disk.
\end{example}

\begin{figure}[htb]
    \centerline{\includegraphics[width=0.75\textwidth]{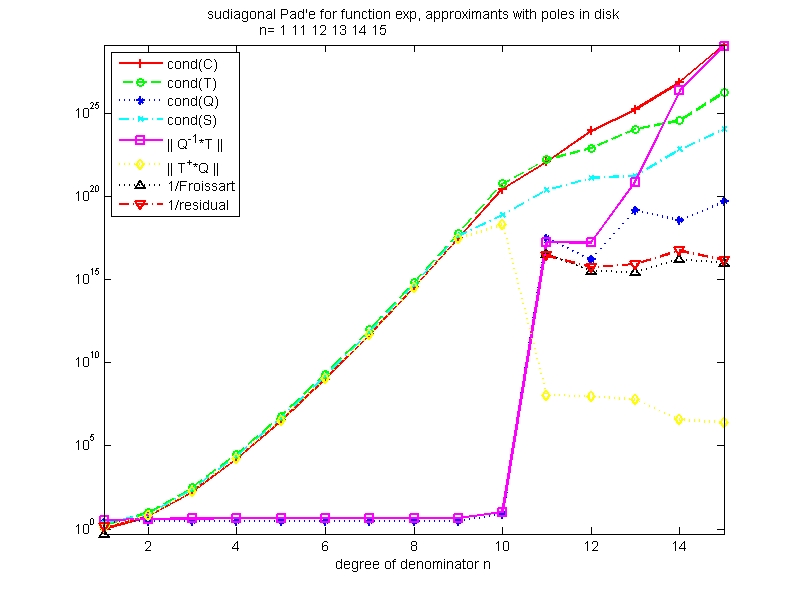}}
  \caption{Condition numbers related to the exponential function $f_2$.}\label{fig2}
\end{figure}

\begin{example}\label{exa2}
   The $[n-1|n]$ Pad\'e approximant for $n=1,...,10$ of the exponential function $f_2$ in \eqref{exa} does not have poles in the (open) unit disk, even in presence of rounding errors. We believe that, due to rounding errors, our $[n-1|n]$ Pad\'e approximants for $n\geq 11$ having poles in the disk  are badly computed. Also, Matlab gives warnings that the condition numbers and norms for $n \geq 11$ are badly computed. We observe from Fig.~\ref{fig2} for $n\leq 10$ that $\kappa(Q)$ is close to $1$, and thus $\kappa(S)$ and $\kappa(T)$ have the same magnitude, which is growing quickly with $n$.
\end{example}

\begin{figure}[htb]
    \centerline{\includegraphics[width=0.75\textwidth]{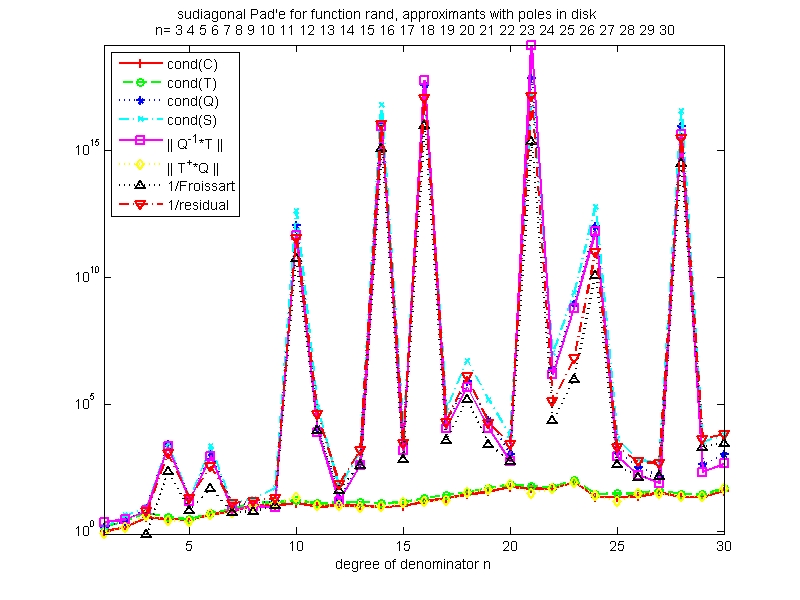}}
  \caption{Condition numbers related to the random function $f_3$.}\label{fig3}
\end{figure}

\begin{example}\label{exa3}
   The numerical results reported in Fig.~\ref{fig3} for the random function $f_3$ in \eqref{exa} for $n=1,2,...,N=30$ depend of course on the realization of the random Taylor coefficients, but for about $10$ realizations we found each time a similar behavior: all approximants are robust since $\kappa(T)$ is always not too far from $1$.
   As a consequence, $\kappa(S)$ and $\kappa(Q)$ have the same magnitude, the dependence on $n$ being quite erratic, in this example between $1$ and $10^{20}$. This shows that there are cases where a Pad\'e approximant is robust but not well-conditioned.
   \\
   Even more striking, in this example the curves for $1/Froissart$ and $1/Residual$ follow quite closely that of $\kappa(S)$, showing that, for this example, Theorem~\ref{thm_rational} is essentially sharp.
\end{example}

In the context of Example~\ref{exa3}, we should also mention the 
recent paper \cite{masc} where, given arbitrary nonzero complex numbers $z_k$ of modulus $\leq 1/3$, the author explicitly gives a function $f$ analytic in $|z|<1$ where the subsequence of $[n_k|n_k]$ diagonal Pad\'e approximants, $n_k=2^k-2$, are robust (with the condition number of $C$ being bounded by $5$) but have a (spurious) pole at $z_k$. His function $f$ is resulting from a smart modification of Gammel's counterexample \cite[\S 6.7]{BGM}, where the $[n_k|n_k]$ approximant coincides with the $[n_k|1]$ Pad\'e approximant, leading to a rich block structure in the Pad\'e table. It can be shown that in this case both $Q$ and $S$ have a condition number being of the same magnitude as $|z_k|^{-2n_k}$, hence these approximants are not well-conditioned.

\section{Conditioning of the Pad\'e map and proof of Theorem~\ref{thm_stability}}\label{sec2}

The aim of this section is to analyze the conditioning of the real Pad\'e map and in particular to provide a proof of Theorem~\ref{thm_stability}.
We start however with two technical statements, the first one relating nondegeneracy 
to the rank of the matrix $S$, and the second relating the smallest and largest singular values of the matrices $C,T,Q$ and $S$.

\begin{lemma}\label{lem_degeneracy}
       Let $p,q$ be two polynomials, $p$ of degree $\leq m$ and $q$ of degree $\leq n$. Then $p/q$ is nondegenerate if and only if the matrix $S$ defined in \eqref{def_Sylvester} has full (row) rank $m+n+1$.
\end{lemma}
\begin{proof}
    Suppose that $p/q$ is degenerate. Then either $p_n=q_m=0$ (implying that the last row of $S$ is zero), or else there exists $\gamma\in \mathbb C$ with $p(\gamma)=q(\gamma)=0$, implying that
    $(1,\gamma,...,\gamma^{m+n})S=0$. Thus, in both cases $S$ does not have full row rank.

    Conversely, suppose that $p/q$ is nondegenerate, then at least one of the leading coefficients $p_m$ or $q_n$ is not vanishing, without loss of generality $p_m\neq 0$.
    Notice that, up to permutation of columns, $S$ equals
    $$
          \left[\begin{array}{ccc}  \underline S & * & * \\ 0 & q_n & - p_m
          \end{array}\right]
    $$
    with the classical square Sylvester matrix $\underline S \in \mathbb C^{(m+n)\times (m+n)}$, obtained from $S$ by dropping the last row, and last column in each column block. With $x_1,...,x_m$ the roots of $p$, observe that by assumption $q(x_j)\neq 0$. We use the formula \cite[Theorem~9.3(ii)]{Labahn}
    $$
         \det \underline S = \pm (p_m)^n \prod_{j=1}^m q(x_j) \neq 0
    $$
    in order to conclude that $\underline S$ and thus $S$ has full row rank.
\end{proof}

Recall from \eqref{S=QT} that $S$ having rank $m+n+1$ implies that the matrix $Q$ defined in \eqref{def_Sylvester} is invertible, and the matrix $T$ defined in \eqref{homogeneous_system2} also has full rank $m+n+1$.

\begin{lemma}\label{lem_T_C}
    Suppose that $S$ has rank $m+n+1$. Then
    for the matrices $C$ of \eqref{homogeneous_system} and $T$ of \eqref{homogeneous_system2} we have that
    \begin{eqnarray}
         &&  \label{T_C1} \max \{ 1 , \| C \| \} \leq \| T \| \leq \sqrt{m+n+2} ,
        \\&& \label{T_C2} \| C^\dagger \|  \leq \| T^\dagger \| \leq \sqrt{2(m+n+2)} \, \| C^\dagger \| .
    \end{eqnarray}
    Furthermore, for the matrices $S,Q$ of \eqref{def_Sylvester} with the normalization
    \eqref{def_Pade_map2} there holds
    \begin{equation} \label{T_C3}
           \| Q \| \leq \sqrt{m+n+1} , \quad \frac{1}{\sqrt{2}} \leq \| S \| \leq \sqrt{m+n+1} , \quad  \| Q^{-1} \| \leq \| T \| \, \| S^\dagger \|
           , \quad  \| T^{\dagger} \| \leq \| Q \| \, \| S^\dagger \|.
    \end{equation}
\end{lemma}
\begin{proof}
    Since $1$ is an entry and $-C$ a submatrix of $T$, we obtain the first inequality of \eqref{T_C1}, and the second follows from the scaling \eqref{scaling} and the general fact that any matrix $\in \mathbb C^{(m+n+1)\times (m+n+2)}$ with columns of norm $\leq 1$ has a Froebenius norm $\leq \sqrt{m+n+2}$.

    For a proof of $\eqref{T_C2}$, we first recall that by assumption both $C$ and $T$ have full row rank, and hence
    $$
          \frac{1}{\| T^\dagger \|} = \min_{y\in \mathbb C^{m+n+1}} \frac{\| y^* T \|}{\| y \|}
          \leq \min_{x\in \mathbb C^{n}} \frac{\| (0,x^*) T \|}{\| x \|} = \min_{x\in \mathbb C^{n}} \frac{\| x^* C \|}{\| x \|} = \frac{1}{\| C^\dagger \|},
    $$
    implying the first inequality. For the second, recall that
    $C^\dagger=C^* (C C^*)^{-1}$ and hence the two matrices
    $$
         T = \left[\begin{array}{cc} I & - L \\ 0 & - C \end{array}\right], \quad
         T^R = \left[\begin{array}{cc} I & - LC^\dagger \\ 0 & - C^\dagger \end{array}\right] = \left[\begin{array}{cc} C & - L \\ 0 & - I \end{array}\right] \left[\begin{array}{cc} C^\dagger & 0 \\ 0 & C^\dagger \end{array}\right]
    $$ satisfy $T T^R = I$. Since the orthogonal projector $T^\dagger T$ is of norm $1$, we conclude that $\| T^\dagger \| = \| T^\dagger T T^R \| \leq \| T^R \|$. It remains to observe that the right-hand factor in the above factorization of $T^R$ has norm $\| C^\dagger \|$, and the left-hand factor has rows of norm $\leq 2$ (in fact $\leq 1$ provided that $n \leq m$) due to \eqref{scaling}.

    We finally turn to a proof of \eqref{T_C3}, the upper bound for $\| Q \|$ following as before from the scaling \eqref{def_Pade_map2}. Using \eqref{def_Pade_map2} we also observe that the sum of the squares of the norms of all columns of the matrix $S$ equals $\| S \|_F^2 = (m+1) \| \mbox{vec}(q) \|^2 +(n+1) \| \mbox{vec}(p) \|^2 \leq m+n+1$
   and the sum of the squares of the norms of the first and $(n+2)$nd column of $S$ equals $\| \mbox{vec}(q) \|^2 + \| \mbox{vec}(p) \|^2=1\leq 2 \, \| S \|^2$, implying the claimed inequalities for $\| S \|$. For the upper bound for $\|Q^{-1}\|$ (which we suspect to be not very sharp), we use \eqref{S=QT} in order to conclude that $I = S S^\dagger=Q T S^\dagger$ and thus $Q^{-1}=T S^\dagger$.
   Finally, since \eqref{S=QT} is a full rank decomposition, we also have that
   $S^\dagger = T^\dagger Q^{-1} $ and thus $T^\dagger = S^\dagger Q$, implying the claimed bound for $\| T^\dagger \|$.
\end{proof}

Let us now turn to a proof of Theorem~\ref{thm_stability}. Here it is helpful to consider the nonlinear map 
$$
      G : \mathbb R^{m+n+2} \ni \widetilde y=\left[\begin{array}{cc} \mbox{vec}(\widetilde p) \\ \mbox{vec}(\widetilde q)
       \end{array}\right] \mapsto \widetilde c = \left[\begin{array}{cc}
 \widetilde c_0\\ \vdots \\\widetilde c_{m+n} \end{array}\right] \in \mathbb R^{m+n+1} , \quad
        \frac{\widetilde p(z)}{\widetilde q(z)} = \sum_{j=0}^{m+n} \widetilde c_jz^j + \mathcal O(z^{m+n+1})_{z\to 0} ,
$$
which is defined at least for pairs of polynomials $\widetilde p,\widetilde q$ with $\widetilde q(0)\neq 0$, as it is true for a neighborhood of any value $F(c)$. As we see below, it will be easier to study the differentiability of $G$ than that of the Pad\'e map $F$.
 Under the assumptions of Theorem~\ref{thm_stability}, we will show by applying the Implicit Function Theorem that $G$ is a kind of local inverse of $F$: there exist neighborhoods $\mathcal W \subset \mathbb R^{m+n+2}$ of $y=F(c)$ and $\mathcal U \subset \mathbb R^{m+n+1}$ of $c$ such that
 \begin{eqnarray}
     && \label{diff1} \mbox{$G$ is differentiable in $\mathcal W$ with Jacobian~}
     J_G(F(c)) = Q^{-1} T,
     \\ && \label{diff2} \mbox{for all $\widetilde y \in \mathcal W\cap \mathbb S^{m+n+2}$ we have that $F(G(\widetilde y ))=\widetilde y$},
     \\ && \label{diff3} \mbox{for all $\widetilde c \in \mathcal U$ we have that $G(F(\widetilde c ))=\widetilde c$} .
 \end{eqnarray}
Then the statement of Theorem~\ref{thm_stability}(a) will follow by setting $\mathcal V=F(\mathcal U)\subset \mathcal W\cap \mathbb S^{m+n+2}$.


\begin{proof}{of Theorem~\ref{thm_stability}(a).}
   Let us first construct a neighborhood $\mathcal W$ of $y=F(c)$ and prove \eqref{diff1}. In the sequel of the proof we adapt the notation $Q=Q(q)$ for the triangular Toeplitz matrix in
   \eqref{def_Sylvester}, and $T_0(c)$ for the submatrix of $T=T(c)$ in \eqref{homogeneous_system2} formed by the last $n+1$ columns.
   First notice that
   $$
        \widetilde y=\left[\begin{array}{cc} \mbox{vec}(\widetilde p) \\ \mbox{vec}(\widetilde q)
       \end{array}\right] , \quad \widetilde c = G(\widetilde y) = Q(\widetilde q)^{-1}
       \left[\begin{array}{cc} \mbox{vec}(\widetilde p) \\ 0
       \end{array}\right] .
   $$
   By assumption and Theorem~\ref{thm_continuity}, $F(c)$ is non degenerate. Thus, by Lemma~\ref{lem_degeneracy}, $S=S(\widetilde y)$ has full row rank for all $\widetilde y\in \mathcal W$, a sufficiently small neighborhood of $y=F(c)$.
   As a consequence, $Q(\widetilde q)$ is invertible, and thus $G$ is well-defined on $\mathcal W$. In addition, by the differentiability of the maps $\mbox{vec}(\widetilde q) \mapsto Q(\widetilde q)$ and $\mbox{vec}(\widetilde q) \mapsto Q(\widetilde q)^{-1}$, we also conclude that $G$ is differentiable on $\mathcal W$. Notice that $\widetilde c=G(\widetilde y)$ does satisfy
   $$
        \left[\begin{array}{cc} \mbox{vec}(\widetilde p) \\ 0
       \end{array}\right]  = Q(\widetilde q) G(\widetilde y) = - T_0(\widetilde c) \mbox{vec}(\widetilde q).
   $$
   Taking the product rule for partial derivatives, we obtain
   $$
        \left[\begin{array}{cc} I & 0 \\ 0 & 0
        \end{array}\right] = Q(\widetilde q) J_G(\widetilde y) - T_0(\widetilde c)
        \left[\begin{array}{cc} 0 & I
        \end{array}\right]
   $$
   implying that $Q(\widetilde q) J_G(\widetilde y)=T(\widetilde c)$, as claimed in \eqref{diff1}.

   We proceed with showing \eqref{diff2}, implying the injectivity of $G$ restricted to $\mathcal W\cap \mathbb S^{m+n+2}$. By definition of $\mathcal W$, we have that $\widetilde y \in \mathcal W\cap \mathbb S^{m+n+2}$ is nondegenerate, in particular $\widetilde q(0)\neq 0$, $\| \widetilde y \|=1$ and trivially $T(\widetilde c)\widetilde y=0$ for $\widetilde c = G(\widetilde y)$ by definition of $G$. Since $q(0)>0$, by possibly making $\mathcal W$ smaller, we may also assume that $\widetilde q(0)>0$. Then $\widetilde y = F(\widetilde c)$ by definition of the Pad\'e map $F$, as claimed in  \eqref{diff2}.

   In order to establish \eqref{diff3} together with the claimed formula for $J_F(c)$, we consider the function
   $$
          H: \mathcal W \times \mathbb R^{m+n+1} \ni (\widetilde y,\widetilde c) \mapsto H(\widetilde y,\widetilde c) = \left[\begin{array}{cc} G(\widetilde y)-\widetilde c \\
          \widetilde y^t \widetilde y - 1 \end{array}\right] ,
   $$
   being of class $\mathcal C^1$ by \eqref{diff1}. Notice that
   $$
        \frac{\partial H}{\partial \widetilde y}(\widetilde y,\widetilde c) = \left[\begin{array}{cc} J_G(\widetilde y) \\ 2 \widetilde y^t\end{array}\right] = \left[\begin{array}{cc} Q(\widetilde q)^{-1} T(\widetilde c) \\ 2 \widetilde y^t\end{array}\right]
   $$
   is invertible since the same is true for
   $$
        \left[\begin{array}{cc} Q(\widetilde q)^{-1} T(\widetilde c) \\ 2 \widetilde y^t\end{array}\right]
        \left[\begin{array}{cc} Q(\widetilde q)^{-1} T(\widetilde c) \\ 2 \widetilde y^t\end{array}\right]^*
        =
                \left[\begin{array}{cc} Q(\widetilde q)^{-1} T(\widetilde c)T(\widetilde c)^*Q(\widetilde q)^{-*} & 0 \\ 0 & 4 \widetilde y^t \widetilde y\end{array}\right]
   $$
   for $\widetilde y \in \mathcal W$ by definition of $\mathcal W$ and for $\widetilde c$ sufficiently close to $c$.
   Also, we have that $H(F(c),c)=0$ because $T(c)F(c)=0$ and $q(0)\neq 0$. The Implicit Function Theorem thus implies the existence of a neigborhood $\mathcal U$ of $c$ and a $\mathcal C^1$ function $\widetilde F:\mathcal U \mapsto \mathcal W\cap \mathbb S^{m+n+2}$ such that
   $H(\widetilde F(\widetilde c),\widetilde c)=G(\widetilde F(\widetilde c))-\widetilde c=0$ for all $\widetilde c\in \mathcal U$, and thus $F(\widetilde c)=F(G(\widetilde F(\widetilde c)))=\widetilde F(\widetilde c)$ by \eqref{diff2}, implying \eqref{diff3}.

   We also learn from the Implicit Function Theorem that
   \begin{eqnarray*}
        J_F(c)&= &- \frac{\partial H}{\partial \widetilde y}(y,c)^{-1} \frac{\partial H}{\partial \widetilde c}(y,c)
        =
        \left[\begin{array}{cc} Q^{-1} T \\ 2 y^t\end{array}\right]^*
                \left[\begin{array}{cc} Q^{-1} TT^*Q^{-*} & 0 \\ 0 & 4 y^t y\end{array}\right]^{-1}
        \left[\begin{array}{cc} I \\ 0
        \end{array}\right]  \\
        &= &T^* (T T^*)^{-1} Q = T^\dagger Q = J_G(F(c))^\dagger.
    \end{eqnarray*}
    To sum up, $F:\mathcal U \mapsto \mathcal V := F(\mathcal U) \subset \mathcal W\cap \mathbb S^{m+n+2}$ is surjective by construction, injective by \eqref{diff3}, differentiable with Jacobian $J_F(c)=T^\dagger Q$, and has the inverse $G|_{\mathcal V}$ being differentiable by \eqref{diff1}, as claimed in Theorem~\ref{thm_stability}(a).
\end{proof}

\begin{proof}{of Theorem~\ref{thm_stability}(b).}
    It is not difficult to check that the neighborhood $\mathcal W$ of $y=F(c)$ constructed above can be chosen to be a ball centered at $y=F(c)$, with radius $r>0$.
    Notice that $F(\mathbb R^{m+n+1})\subset \mathbb S^{m+n+2}$, and thus for $\widetilde y \in \mathcal W$
    $$
         \mbox{dist}(\widetilde y,F(\mathbb R^{m+n+1}))
         \leq \mbox{dist}(\widetilde y,\mathbb S^{m+n+2}) =
         \| \widetilde y - \frac{\widetilde y}{\|\widetilde y\|}\|  = \Bigl| \| \widetilde y \| - 1 \Bigr|.
    $$
    Thus for establishing the statement of Theorem~\ref{thm_stability}(b) it only remains to show that $\widetilde y/\|\widetilde y\| \in F(\mathbb R^{m+n+1})$, which would follow from \eqref{diff2} provided that $\widetilde y/\|\widetilde y\|\in \mathcal W$. In order to show the latter, notice that $\| y \|=1$, and thus
    $$
         \| y - \frac{\widetilde y}{\|\widetilde y\|} \| \leq \| y - \widetilde y \| + \Bigl| \, \| \widetilde y \| - \| y \| \, \Bigr| < r
    $$
    for $\widetilde y$ sufficiently close to $y$, and thus ${\widetilde y}/{\|\widetilde y\|}\in \mathcal W$.
\end{proof}

\begin{proof}{of Theorem~\ref{thm_stability}(c).}
    From \cite{TrefethenBau} we have the following well-known relation for the forward condition number $\kappa_{for}(F)(c)$
    $$
       \limsup_{\widetilde c \to c}
       \frac{{\| F(\widetilde c) - F(c) \|}}{{\| \widetilde c - c \|}} =
       \limsup_{\widetilde c \to c}
       \frac{{\| F(\widetilde c) - F(c) \|}/{\|F(c) \|}}{{\| \widetilde c - c \|}/{\|c \|}}
       = \frac{\| c \|}{\| F(c)\|} \| J_F(c) \| = \| T^\dagger Q  \|,
    $$
    where we used the facts that $\| c \|=1$ according to \eqref{scaling}, $\| F(c) \|=1$ by definition \eqref{def_Pade_map2}, and that we have the explicit formula of Theorem~\ref{thm_stability}(a) for the Jacobian.
\end{proof}

\begin{proof}{of Theorem~\ref{thm_stability}(d).}
   From the proof of Theorem~\ref{thm_stability}(b) and \eqref{diff2} we know that,
   for $\widetilde y$ sufficiently close to $y=F(c)$,
   \begin{equation} \label{diff4}
         \mbox{dist}(\widetilde y,F(\mathbb R^{m+n+1})) = \| \widetilde y - F(\widetilde c) \|,
   \end{equation}
   with $\widetilde c=G(\widetilde y/\| \widetilde y\|)$. Notice that, by \eqref{diff3}, there are no other arguments $\widetilde c\in \mathcal U$ satisfying \eqref{diff4}. Also,
   $\widetilde c=G(\widetilde y)$ by definition of $G$. Thus
   $\inf \{ \| \widetilde c - c \| : F(\widetilde c)=\widetilde y/\| \widetilde y\|\} = \| G(\widetilde y) - G(y) \|$, and $$
       \kappa_{back}(F)(c)=\kappa_{for}(G)(F(c))= \| J_G(F(c)) \| = \| Q^{-1} T \|,
   $$
   where in the last equality we applied \eqref{diff1}.
\end{proof}


\section{Distances between two rational functions and their coefficient vectors}\label{sec3}

A central question in this paper is how to measure the distance between
two rational functions
$$
      r = p/q \in \mathcal R_{m,n}, \quad \widetilde r = \widetilde p/\widetilde q \in \mathcal R_{m,n},
$$
with coefficient vectors
$$
   x(r) = \left[\begin{array}{cc} \mbox{vec}(p) \\\mbox{vec}(q)
        \end{array}\right], \quad
   x(\widetilde r) = \left[\begin{array}{cc} \mbox{vec}(\widetilde p) \\\mbox{vec}(\widetilde q)
        \end{array}\right].
$$
A natural metric in the set $\mathcal M_K$ of functions meromorphic in some compact $K\subset \mathbb C$ would be the uniform chordal metric $\chi_K(r,\widetilde r)$ introduced in \eqref{chordal}.
This metric is well adapted to study uniform convergence questions, since meromorphic functions are continuous on the Riemann sphere. We will also see that it enables us to study Froissart doublets and small residuals. However, it is not so clear how to relate such a metric to the coefficient vectors in the basis of monomials of numerators and denominators of rational functions, which are used to parametrize rational functions in the Pad\'e map.
This is essentially due to the fact that there are several coefficient vectors $x(r)$ representing the same rational function $r$: even if we suppose that $r$ is nondegenerate, we still may multiply $x(r)$ by an arbitrary complex scalar. As before, we will always suppose that coefficient vectors are of norm $1$, but this fixes only the absolute value but not the phase of the scalar normalization constant. For defining a metric between rational functions it will therefore be suitable to measure the distance of coefficient vectors with optimal phase
\begin{equation} \label{dist}
     r,\widetilde r \in \mathcal R_{m,n}:
     \quad d(r,\widetilde r) := \min \{ \| x(r) - a x(\widetilde r)\|: a \in \mathbb C, |a|=1 \}.
\end{equation}
The reader easily checks that $\| x(r) - a x(\widetilde r)\|$ does not depend on $a$ if
$x(r)$ and $x(\widetilde r)$ are mutually orthogonal, and else
\begin{equation} \label{dist2}
      \arg\min \{ \| x(r) - a x(\widetilde r)\|: a \in \mathbb C, |a|=1 \}
      = \frac{x(\widetilde r)^*x(r)}{|x(\widetilde r)^*x(r)|}.
\end{equation}
In particular, if both $x(r)$ and $x(\widetilde r)$ are real then $$
    d(r,\widetilde r) = \min \{ \| x(r) - x(\widetilde r)\|, \| x(r) + x(\widetilde r)\| \},
$$
and more precisely $d(r,\widetilde r) = \| x(r) - x(\widetilde r)\|$ provided that $x(\widetilde r)^*x(r)\geq 0$ or $\| x(r) - x(\widetilde r)\|\leq \sqrt{2}$, as it was the case in our study of the continuity and the conditioning of the real Pad\'e map.

Recall from the introduction that we called a rational function $r=p/q$ well-conditioned if the condition number $\kappa(S)$ is not too large, $\kappa(S)$ not depending on the normalization of the coefficient vector occurring in \eqref{def_Sylvester}. The following result shows that the two distances $d(r,\cdot)$ and $\chi_{\mathbb D}(r,\cdot)$ for the closed unit disk $\mathbb D$ introduced above are of comparable size provided that $r$ is well-conditioned.

\begin{theorem}\label{thm_distance}
   Let $r=p/q$ be nondegenerate, then for all $\widetilde r\in \mathcal R_{m,n}$
   \begin{equation} \label{distance}
           \frac{(m+n+1)^{-3/2}}{\sqrt{2} \kappa(S)}
           \, d(r,\widetilde r)
           \leq
           \chi_{\mathbb D}(r,\widetilde r)
           \leq \sqrt{2(m+n+1)} \, \kappa(S)
           \, d(r,\widetilde r)
           .
   \end{equation}
\end{theorem}
\begin{proof}
   According to \eqref{dist}, \eqref{dist2}, and our convention on the norm we may choose the phase of $x(\widetilde r)$ such that
   \begin{equation} \label{distance_normalization}
       x(r)= \left[\begin{array}{cc} \mbox{vec}(p) \\\mbox{vec}(q)
       \end{array}\right] , \quad
       x(\widetilde r)= \left[\begin{array}{cc} \mbox{vec}(\widetilde p) \\\mbox{vec}(\widetilde q)
       \end{array}\right]  \quad \mbox{are of norm $1$,} \quad
       \| x(r) - x(\widetilde r) \| = d(r,\widetilde r)
   \end{equation}
   and hence $x(r)^* x(\widetilde r)\geq 0$.
   Hence we may repeat the arguments in the proof of \eqref{T_C3} and get the inequalities
   \begin{equation} \label{distance_norm}
       {1}/{\sqrt{2}} \leq \| S \| \leq \sqrt{m+n+1} .
   \end{equation}

   In order to establish the right-hand inequality of \eqref{distance},
   it is sufficient to show the relation
   \begin{equation} \label{distance1}
         z\in \mathbb D : \quad
         \chi(r(z),\widetilde r(z)) \leq \sqrt{2(m+n+1)} \, \kappa(S)  \, \| x(r) - x(\widetilde r) \|.
   \end{equation}
   By definition of the chordal metric and the Cauchy-Schwarz inequality,
   \begin{eqnarray}
      \chi(r(z),\widetilde r(z))  &=& \nonumber
      \frac{|(p(z)-\widetilde p(z))\widetilde q(z) - \widetilde p(z)(q(z)-\widetilde q(z))|}{\sqrt{|p(z)|^2 + |q(z)|^2}\sqrt{|\widetilde p(z)|^2 + |\widetilde q(z)|^2}}
             \leq \frac{\Bigl\| \left[\begin{array}{cc} p(z)-\widetilde p(z) \\ q(z)-\widetilde q(z)\end{array}\right]\Bigr\|}{\sqrt{|p(z)|^2 + |q(z)|^2}}
      \\
      &=& \label{distance2}
      \frac{\Bigl\|  \left[\begin{array}{cc} 1,z,...,z^m & 0 \\ 0 & 1,z,...,z^n \end{array}\right]
 (x(r)-x(\widetilde r))\Bigr\| }{\sqrt{|p(z)|^2 + |q(z)|^2}}.
   \end{eqnarray}
   Let us study separately the term in the denominator. We remark that
   \begin{equation} \label{distance3}
         \Bigl(1,z, \cdots z^{n+m}\Bigr) S = \Bigl(-q(z), -zq(z),\cdots ,-z^mq(z), p(z), \cdots , z^np(z)\Bigr).
   \end{equation}
   By Lemma~\ref{lem_degeneracy} we know that the Sylvester-like matrix $S$ has full row rank and hence $S S^\dagger =I$. Multiplying the above relation on the right by $S^\dagger$ and taking norms we arrive at
   \begin{eqnarray*}
         \Bigl\| \Bigl(1,z, \cdots z^{n+m}\Bigr) \Bigr\|^2
         &\leq & \| S^\dagger \|^2 \,  \Bigl\| \Bigl(-q(z), -zq(z),\cdots ,-z^mq(z), p(z), \cdots , z^np(z)\Bigr) \Bigr\|^2
         \\
         &\leq & \| S^\dagger \|^2 \,  \Bigl\| \left[\begin{array}{cc} 1,z,...,z^m & 0 \\ 0 & 1,z,...,z^n \end{array}\right] \Bigr\|^2 \, ( |p(z)|^2+|q(z)|^2 ),
   \end{eqnarray*}
   which implies that
   \begin{equation}\label{distance4}
       \forall \, z \in \mathbb C : \quad 1 \leq \| S^\dagger \| \, \sqrt{|p(z)|^2 + |q(z)|^2}  .
   \end{equation}
   Inserting \eqref{distance4} into \eqref{distance2} and using \eqref{distance_norm} and the fact that $z \in \mathbb D$ implies \eqref{distance1}.

   It remains the left-hand inequality of \eqref{distance},
   for which it is sufficient to show
   \begin{equation} \label{distance11}
           d(r,\widetilde r)
           \leq
           \sqrt{2} \, (m+n+1)^{3/2}\, \kappa(S)  \, \chi_{K}(r,\widetilde r),
   \end{equation}
   with $K$ the set of $(m+n+1)$th roots of unity
 $\xi_j=e^{(2i\pi j)/(m+n+1)}$, $j=0,\cdots , m+n$.
   Denote by $\Omega=(\frac{1}{\sqrt{m+n+1}}\xi_j^k)_{j,k=0,...,m+n}$ the unitary DFT matrix of order $m+n+1$.
   A simple computation shows that $S x(r)=0$. Since  Lemma~\ref{lem_degeneracy} shows that the kernel of $S$ has dimension one and $\| x(r) \| =1$, we have $S^\dagger S = I - x(r) x(r)^*$. Since $x(r)^* x(\widetilde r)\geq 0$, we find an angle $\alpha\in (0,\pi/2]$ such that $\cos(\alpha)=x(r)^* x(\widetilde r)/(\| x(r) \| \, \| x(\widetilde r) \|)=x(r)^* x(\widetilde r)$. Thus
   $$
       d(r,\widetilde r)=
       \| x(r) - x(\widetilde r) \| = \sqrt{2-2\cos(\alpha)} = 2 \sin(\alpha/2),
   $$
   whereas
   $$
       \| S^\dagger S  (x(r) - x(\widetilde r)) \|
       =        \| x(\widetilde r) - x(r) \cos(\alpha)  \|
       = \sqrt{1-\cos^2(\alpha)} = \sin(\alpha) = 2 \sin(\alpha/2) \cos(\alpha/2).
   $$
   Thus $\| S^\dagger S  (x(r) - x(\widetilde r)) \|=\cos(\alpha/2) \, d(r,\widetilde r) \geq d(r,\widetilde r)/\sqrt{2}$, implying that
   $$d(r,\widetilde r)/\sqrt{2} \leq
       \| S^\dagger S (x(r) - x(\widetilde r)) \| \leq
       \| S^\dagger \| \, \| S (x(r) - x(\widetilde r)) \|
       =   \| S^\dagger \| \, \| \Omega S (x(r) - x(\widetilde r)) \| ,$$
      where the last equality follows from the orthogonality of $\Omega $. The $j$th entry of $\Omega S\left(x(r)-x(\widetilde r)\right)$ equals the $j$th entry of  $-\Omega S \left( x(\widetilde r\right)$, which in turn is equal to $(p(\xi_j)\widetilde q(\xi_j )-\widetilde p(\xi_j) q (\xi_j)/\sqrt{m+n+1}$, and so
       \begin{equation}\label{distance12}
       d(r,\widetilde r)/\sqrt{2} \leq \| S^\dagger \| \,  \max_{z\in K} |p(z) \widetilde q(z) - q(z) \widetilde p(z)|
       \end{equation}
   Returning to \eqref{distance3}, we also find that
   \begin{eqnarray}\label{distance13}
       \forall \, |z|\leq 1  : \quad && (m+n+1) \| S \|^2  \geq \| (1,z,...,z^{m+n}) S \|^2
        \\ && \nonumber
        =  |p(z)|^2\, \sum_{j=0}^n |z|^{2j} + |q(z)|^2\, \sum_{j=0}^m |z|^{2j} \geq |p(z)|^2 + |q(z)|^2  .
   \end{eqnarray}
   A similar bound is obtained for $\widetilde p(z),\widetilde q(z)$, which combined with \eqref{distance_norm} becomes $$
           \forall \, |z|\leq 1  : \quad (m+n+1) \geq \sqrt{|\widetilde p(z)|^2 + |\widetilde q(z)|^2}  .
   $$
   Inserting these two relations into the right-hand side of \eqref{distance12} implies \eqref{distance11}.
\end{proof}

\section{Proofs of Theorem~\ref{thm_rational} and of Theorem~\ref{thm_convergence}}\label{sec4}

We start by establishing a technical result on the condition number of Sylvester-like matrices close to $S$.

\begin{lemma}\label{lem_conditioning}
   Let $r=p/q$ be nondegenerate. If $\widetilde r=\widetilde p/\widetilde q\in \mathcal R_{m,n}$ satisfies
   \begin{equation} \label{eq_conditioning}
       \sqrt{2(m+n+1)} \, d(r,\widetilde r) \, \kappa(S)  \leq 1/3,
   \end{equation}
   then it is nondegenerate, and $\kappa(\widetilde S)\leq 2 \, \kappa(S)$ for the Sylvester-like matrix $\widetilde S=S(-\widetilde q,\widetilde p)$ constructed as in \eqref{def_Sylvester}.

   More generally, if $\widetilde r$ is degenerate then $\sqrt{2(m+n+1)} \, d(r,\widetilde r) \, \kappa(S)\geq 1$.
\end{lemma}
\begin{proof}
   For a proof of the first statement, write $E:=S^\dagger (S-\widetilde S)$, and denote by $x(r),x(\widetilde r)$ the corresponding coefficent vectors with unit norm and particular phase such that $\| x(r) - x(\widetilde r) \|= d(r,\widetilde r)$.  Then
   $S(I-E)=S - SS^\dagger (S-\widetilde S)=\widetilde S$. Using the same arguments as in the proof of \eqref{distance_norm}, we obtain
   \begin{eqnarray*}
        \| E \| &\leq & \| S^\dagger \| \, \| S-\widetilde S \|
        \leq \sqrt{m+n+1} \| S^\dagger \| \, \| x(r)-x(\widetilde r) \|
        \\&\leq& \sqrt{2(m+n+1)} \, \kappa(S) \, d(r,\widetilde r) \leq 1/3
   \end{eqnarray*}
   by assumption \eqref{eq_conditioning} on $\widetilde r$. Hence $\| \widetilde S \| \leq (1+\| E \|)\,  \| S \|\leq \frac{4}{3} \, \| S \|$.
   Also, $(I-E)^{-1} S^\dagger$ is a right inverse of $\widetilde S$, showing that $\widetilde S$ has full row rank, and that
   $$
          \| \widetilde S^\dagger \| = \| \widetilde S^\dagger \widetilde S (I-E)^{-1} S^\dagger \| \leq \| (I-E)^{-1} \| \, \| S^\dagger \| \leq \frac{3}{2} \| S^\dagger \|,
   $$
   from which the first assertion follows.

   For the second part, we know from Lemma~\ref{lem_degeneracy} that $\rank \widetilde S <m+n+1$, and hence for the smallest singular value of $S$ by the Eckhard-Young Theorem
   $$
        \frac{1}{\| S^\dagger \|} = \sigma_{m+n+1}(S) \leq \| S - \widetilde S \|
        \leq \sqrt{2(m+n+1)} \, \| S \| \, d(r,\widetilde r) ,
   $$
   as claimed above.
\end{proof}

We are now prepared to proceed with a proof of Theorem~\ref{thm_rational}.

\begin{proof}{of Theorem~\ref{thm_rational}(a).\, }
  Let $\sigma,\tau \in \mathbb D$ with $\widetilde r(\sigma)=0$, $\widetilde r(\tau)=\infty$, then $\chi(r(\tau),r(\sigma))\geq 1/3$ because of
  \begin{eqnarray*}
     1 &=& \chi(\widetilde r(\tau),\widetilde r(\sigma)) \leq
     \chi(r(\tau),r(\sigma))  + \chi(\widetilde r(\tau),r(\tau)) + \chi(r(\sigma),\widetilde r(\sigma))
     \\&\leq &
     \chi(r(\tau),r(\sigma)) + 2 \, \chi_{\mathbb D}(r,\widetilde r) \leq
     \chi(r(\tau),r(\sigma)) + \frac{2}{3} .
  \end{eqnarray*}
  Consider the spherical derivative
  \begin{equation} \label{rational1}
      r^{\#}(z) := \frac{|r'(z)|}{1+|r(z)|^2} .
  \end{equation}
  We claim that
  \begin{equation} \label{rational2}
     \frac{\chi(r(\tau),r(\sigma))}{|\tau - \sigma|} \leq
     \max_{z\in \mathbb D} r^{\#}(z) \leq \sqrt{2} (m+n+1)^{3/2} \, \kappa(S)
  \end{equation}
  which implies that $|\tau - \sigma|\geq 1/(3 \sqrt{2} (m+n+1)^{3/2} \, \kappa(S))$, as claimed in Theorem~\ref{thm_rational}.

  In order to show the left-hand inequality of \eqref{rational2}, recall from
  \cite{Sch} that the chordal metric is dominated by
  $$
      \forall \, w_1,w_2\in \overline{\mathbb C}:  \quad \chi(w_1, w_2) \leq \int_\gamma \frac{|dw|}{1+|w|^2},
  $$
  where $\gamma$ is any differentiable curve in the extended complex plane joining $w_1$ with $w_2$. Taking $\gamma: \mathbb D \supset [\sigma,\tau]\ni z \mapsto r(z)\in \overline{\mathbb C}$, we conclude that
  $$
      \chi(r(\sigma),r(\tau)) \leq \int_\gamma \frac{|dw|}{1+|w|^2} =
      \int_{z\in [\sigma,\tau]}  r^{\#}(z) \, |dz| \leq |\sigma-\tau| \, \max_{z\in \mathbb D} r^{\#}(z) \, ,
  $$
  as claimed above. It remains to give an upper bound for $r^{\#}(z)$ for $z\in \mathbb D$, here we closely follow arguments of the proof of Theorem~\ref{thm_distance}.
  We have
  \begin{eqnarray*}
        r^{\#}(z) &=& \frac{| |p'(z)q(z)-q'(z)p(z)|}{|p(z)|^2+|q(z)|^2}
        \leq \frac{\Bigl\| \left[\begin{array}{cc} p'(z)\\ q'(z)\end{array}\right] \Bigr\|}{\sqrt{|p(z)|^2+|q(z)|^2}} \leq \| S^\dagger \| \, \Bigl\| \left[\begin{array}{cc} p'(z)\\ q'(z)\end{array}\right] \Bigr\|,
  \end{eqnarray*}
  where in the last step we have applied \eqref{distance4}. Since
  $$
       \Bigl\| \left[\begin{array}{cc} p'(z)\\ q'(z)\end{array}\right] \Bigr\|
       = \Bigl\| \left[\begin{array}{cc} 0,1,2z,...,m z^{m-1} & 0 \\ 0 & 0,1,2z,...,nz^{n-1} \end{array}\right] x(r) \Bigr\|
       \leq (m+n+1)^{3/2}
  $$
  and $1 \leq \sqrt{2} \, \| S \|$ by \eqref{distance_norm}, we obtain the second inequality claimed in \eqref{rational2}, and hence the part of Theorem~\ref{thm_rational} on Froissart doublets is shown.
\end{proof}

\begin{proof}{of Theorem~\ref{thm_rational}(b).\, }
   We start by observing that for the residual $\alpha_0$ of a simple pole $z_0\in \mathbb D$ of $r=p/q\in \mathcal R_{m,n}$ there holds
   $$
        \frac{1}{|\alpha_0|} = \frac{|q'(z_0)|}{|p(z_0)|} = r^{\#}(z_0)
        \leq \sqrt{2} (m+n+1)^{3/2} \, \kappa(S)
   $$
   where for the last inequality we have applied \eqref{rational2}.
   The assumption
   $2 (m+n+1)^2 \kappa(S) \chi_{\mathbb D}(r,\widetilde r)\leq 1/3$ together with Theorem~\ref{thm_distance} tells us that \eqref{eq_conditioning} holds, and thus also $\widetilde r$ is nondegenerate. By applying the same reasoning as for $r$, we obtain
   for the residual $\widetilde \alpha_0$ of a simple pole $\widetilde z_0\in \mathbb D$ of $\widetilde r=\widetilde p/\widetilde q\in \mathcal R_{m,n}$
   the claimed inequality
   $$
        \frac{1}{|\widetilde \alpha_0|}
        \leq \sqrt{2} (m+n+1)^{3/2} \, \kappa(\widetilde S)
        \leq 2\sqrt{2} \, (m+n+1)^{3/2} \, \kappa(S)
   $$
   where for the last inequality we have applied the first part of Lemma~\ref{lem_conditioning}.
\end{proof}

\begin{remark}\label{rem_neighborhood}
    {\rm
    Recall from the above proof of Theorem~\ref{thm_rational}(b) that we have shown the lower bound $1/((2(m+n+1))^{3/2} \kappa(S))$ for the modulus of any residual of a simple pole in the unit disk of any $\widetilde r\in \mathcal R_{m,n}$ solely under the hypothesis $\sqrt{2(m+n+1)} \, d(r,\widetilde r) \, \kappa(S)  \leq 1/3$, which according to Theorem~\ref{thm_distance} is weaker than the hypothesis
    $2 \, (m+n+1)^{2} \kappa(S)^2 \chi_{\mathbb D}(r,\widetilde r)\leq 1/3$ stated in Theorem~\ref{thm_rational}(b), and stronger than the hypothesis $\chi_{\mathbb D}(r,\widetilde r)\leq 1/3$ of Theorem~\ref{thm_rational}(a).

    In the numerical procedure described in \cite{GGT12}, the authors do not necessarily return the $[m|n]$ Pad\'e approximant $r=p/q$ but $\widetilde r=\widetilde p/\widetilde q$ obtained by replacing the $\ell$ leading coefficients of $p$ (or of $q$, but not of both since otherwise $\kappa(S)$ would be large) of modulus $\leq \epsilon$ by $0$. Thus $d(r,\widetilde r)\leq \| x(r) - x(\widetilde r) \|\leq \sqrt{\ell} \epsilon$, and Theorem~\ref{thm_rational}(a),(b) do apply provided that $\sqrt{2\ell(m+n+1)} \, \epsilon \kappa(S) \leq 1/3$.}
\end{remark}

\begin{remark}\label{rem_residual}
   {\rm
   By examining the above proofs and using elementary techniques of complex analysis we see that it is possible to generalize Theorem~\ref{thm_rational} to the case $r,\widetilde r\in \mathcal M(\mathbb D)$ of general meromorphic functions (at least if $r$ has no zeros/poles on the unit circle), but the price to pay is that the constants become less explicit, in particular there is no longer the condition number of a matrix.

   For instance, by examining the proof of Theorem~\ref{thm_rational}(a) we see that we can give a lower bound for the
   Euclidian
   distance between a pole and a zero of $\widetilde r$ in terms of the reciprocal of the maximum spherical derivative of $r$ on the unit disk $\mathbb D$ provided that $\chi_{\mathbb D}(r,\widetilde r)\leq 1/3$. Moreover, from the Rouch\'e Theorem we see that for any sufficiently small $\epsilon>0$ there exists a (computable) $\delta>0$ depending on $r$ and $\epsilon$ such that, for any $\widetilde r \in \mathcal M(\mathbb D)$ with $\chi_{\mathbb D}(r,\widetilde r)\leq \delta$ we have that
   the $\epsilon$-neighborhood of any pole or zero of $r$ contains the same number of poles or zeros of $\widetilde r$ counting multiplicities as $r$, and $\widetilde r$ has no other poles and zeros in $\mathbb D$.
   This constitutes an alternative approach to control Froissart doublets of $\widetilde r$.

   In addition, by possibly choosing a smaller $\delta>0$ we may insure that, for a simple pole of $r$, the residual of the corresponding simple pole of $\widetilde r$ differs from that of $r$ at most by $\epsilon$, giving a possibility to exclude small residuals for $\widetilde r$. Thus we may roughly summarize by saying that if $\chi_{\mathbb D}(r,\widetilde r)$ is sufficiently small then $r$ has a spurious pole if and only if $\widetilde r$ has.
}
\end{remark}

\begin{proof}{of Theorem~\ref{thm_convergence}.\, }
   By assumption
   and the second part of Lemma~\ref{lem_conditioning}
   $$
     \sqrt{2(m+n+1)} \, d(r,\widetilde r) \, \kappa(S)\geq 1,
   $$
   and a combination with Theorem~\ref{thm_distance} implies that
   $
           2 \, (m+n+1)^2 \, \chi_{\mathbb D}(r,\widetilde r) \, \kappa(S)^2\geq 1,
   $
   as claimed in Theorem~\ref{thm_convergence}.
\end{proof}

\section{Numerical GCD and other related results}\label{sec5}

\subsection{Froissart doublet and numerical GCD}\label{sec5.1}
One could wonder whether the existence of Froissart doublets of a rational function $r=p/q\in \mathcal R_{m,n}$, namely the existence of a zero $\sigma$ and a pole $\tau$ of $r$ with small
Euclidian
distance $|\sigma-\tau|$, is related to the fact that the pair $(p,q)$ is close to a similar pair $(\widetilde p,\widetilde q)$ with non-trivial greatest common divisor (GCD), or more generally being degenerate, that is, the quantity
$$
       \epsilon(p,q):=\min \Bigl\{ \Bigl\| x(r) - \left[\begin{array}{cc}
       \mbox{vec}(\widetilde p) \\ \mbox{vec}(\widetilde q) \end{array}\right]
       \Bigr\| : \widetilde p/\widetilde q \in \mathcal R_{m,n} \mbox{~is degenerate} \Bigr\}
$$
is small. This quantity 
has been discussed in \cite{BeLa}. According to \cite[Theorem~4.1 and Remark~4.3]{BeLa} we have
\begin{equation} \label{eps_formula}
       \epsilon(p,q) = \inf_{z\in \overline{\mathbb C}} \sqrt{\frac{|p(z)|^2}{1+|z|^2+...+|z|^{2m}}
       + \frac{|q(z)|^2}{1+|z|^2+...+|z|^{2n}}},
\end{equation}
the argument $z^*$ where the infimum is attained being called the closest common root (which is indeed a common root of the closest degenerate pair).
The following link between numerical GCD and Froissart doublets has been claimed without proof in \cite[Section 4]{BeLa}. For the sake of completeness we give here a proof.

\begin{lemma}\label{ngcd}
Let $\tau, \sigma \in\mathbb{D}$ satisfy $p(\sigma )=0$ and $q(\tau )=0$. Then
\begin{equation}\label{relNGCD}
\epsilon (p,q) \leq \min\conj{m,n} \mo{\sigma -\tau}.
\end{equation}
\end{lemma}
\begin{proof}
Since $\mo{\sigma}\leq 1$,  $\mo{\tau}\leq 1$, $\| \mbox{vec}(p)\|\leq 1$, using  twice the
Cauchy-Schwarz  inequality we obtain
\begin{eqnarray*}
\mo{p(\tau )} & = & \mo{p(\tau )-p (\sigma )}=\mo{\sum_{k=1}^m p_k (\tau^k-\sigma^k)}\leq \sum_{k=1}^m\mo{p_k}\mo{\tau^k-\sigma^k}\\
&=&\mo{\tau -\sigma}\sum_{k=1}^m \mo{p_k}\mo{\sum_{i=0}^{k-1}\tau^i\sigma^{k-i-1}}\leq \mo{\tau -\sigma}\sum_{k=1}^m\mo{p_k}\left(\sum_{i=0}^{k-1}\mo{\tau}^i\right)\\
\\
&\leq & \mo{\tau -\sigma}m\sqrt{\sum_{i=0}^{2m}\mo{\tau}^i}.
\end{eqnarray*}
Using a similar argument for $q(\sigma)$ and replacing in \eqref{eps_formula},  the claimed inequality \eqref{relNGCD} follows.
\end{proof}

\subsection{Numerical GCD and structured smallest singular values}\label{sec5.2}
Recall from Lemma~\ref{lem_degeneracy} that $\widetilde r=\widetilde p/\widetilde q$ is degenerate if and only if the corresponding Sylvester-like matrix $\widetilde S$ is not of full rank. According to the arguments in the proof of, e.g., \eqref{distance_norm} or Lemma~\ref{lem_conditioning}, the expression $\| x(r) - x(\widetilde r) \|$ in the definition of $\epsilon(p,q)$ can be replaced, up to some modest power of $(m+n+1)$, by
$\| S - \widetilde S \|$ or by $\| S - \widetilde S \|/\| S \|$. In other words, $\epsilon(p,q)$ is essentially the absolute or relative distance of $S$ to the set of not full rank Sylvester-like matrices, a kind of smallest structured singular value of $S$, or reciprocal structured condition number.
Since the distance to the set of all not full rank matrices is smaller, we get from the Eckhard-Young Theorem that $\frac{1}{\kappa(S)} \lesssim \epsilon(p,q)$, which is essentially the finding of the second part of Lemma~\ref{lem_conditioning}. In particular, the inequality $\epsilon(p,q) \lesssim |\sigma - \tau|$ of Lemma~\ref{ngcd} implies ${1} \lesssim \kappa(S) \, |\sigma - \tau|$, a result which is established rigorously in Theorem~\ref{thm_rational}(a).

We should mention the relation with \cite{BeLa,CGTW} who both do not argue in terms of our matrix $S$ defined in \eqref{def_Sylvester} but in terms of the classical square Sylvester matrix $\underline S$ of order $m+n$ obtained from $S$ by dropping the last column in each column block and the last 
row. However, we believe that this difference is not essential.
In \cite{CGTW} one looks at a gap in the singular values of $\underline S$ in order to find the degree of a numerical GCD, in particular, (normalized) pairs $(p,q)$ of polynomials with sufficiently ``large'' $\sigma_{m+n}(\underline S)\sim 1/\kappa(\underline S)$ should be considered as numerically coprime. This has to be compared with our notion of well-conditioned rational functions where $\kappa(S)$ is modest. While working with different vector norms, the authors in \cite{BeLa} 
introduce the estimator $$
       \kappa_{BL}:=\max(\| \underline S^{-1} e_1\|,\| \underline S^{-1}e_{m+n}\|),
$$
$e_j$ denoting the $j$th canonical vector, and show that $1/\kappa_{BL}\lesssim \epsilon(p,q)$, and $\sqrt{\kappa(\underline S)}\lesssim \kappa_{BL} \leq \kappa(\underline S)$.

Extending the arguments of \cite{BeLa}, we get the following sharper complement of Theorem~\ref{thm_convergence}.

\begin{lemma}\label{lem_convergence}
   For the nondegenerate $[m|n]$ Pad\'e approximant $r=p/q$ and the (possibly degenerate)
   $[m-1|n-1]$ Pad\'e approximant $\widetilde r = u/v$ we have that
   for all $|z|\leq 1$
   $$
              \kappa \, \chi(r(z),\widetilde r(z)) \geq |z|^{m+n-1} , \quad
              \kappa:=\min\{ 2 (m+n+1)^{{3}/{2}} \,\kappa(\underline{S}) ,
              (m+n+1)^2\, \kappa_{BL} \}.
   $$
\end{lemma}
\begin{proof}
   Notice that $\underline S^{-1} e_{n+m}$ is a not normalized coefficient vector of the rational function $u/v\in \mathcal R_{m-1,n-1}$ satisfying $q(z)u(z)-p(z)v(z)=z^{m+n-1}$ and hence $$
     q(z)(f(z)v(z) - u(z))=v(z)(f(z)q(z) - p(z))+\mathcal O(z^{m+n-1})_{z\to 0}=\mathcal O(z^{m+n-1})_{z\to 0}.
   $$ Then the relation $q(0)\neq 0$ implies that $\widetilde r=u/v$ is the $[m-1|n-1]$ Pad\'e approximant of $f$.

   Writing in this proof $\widehat S\in \mathbb C^{(m+n-1)\times(m+n)}$ for the Sylvester-like matrix of $(u,v)$, we find as in the proof of \eqref{T_C3} that $\| \widehat S \| \leq \sqrt{m+n+1} \, \| \underline S^{-1} e_{n+m} \| \leq \sqrt{m+n+1} \, \kappa_{BL} \leq \sqrt{m+n+1} \, \| \underline S^{-1} \|$. We also have that $\| \underline S \| \leq \| S \| \leq \min \{ \sqrt{m+n+1}, 2 \| \underline S \| \}$ since one is a submatrix of the other. It follows that
   $\kappa \geq (m+n+1) \| S \| \, \| \widehat S \|$.
   Consequently, for all $|z|\leq 1$,
   \begin{eqnarray*}
            \kappa \, \chi(r(z),\widetilde r(z))
            &\geq&
            (m+n+1) \| S \| \, \| \widehat S \| \, \chi(r(z),\widetilde r(z))
            \\&\geq&
            |z|^{m+n-1} \,
            \frac{\sqrt{m+n+1}\, \| S \|}{\sqrt{|p(z)|^2+|q(z)|^2}}
            \frac{\sqrt{m+n+1}\, \| \widehat S \|}{\sqrt{|u(z)|^2+|v(z)|^2}} \geq |z|^{m+n-1} ,
   \end{eqnarray*}
   where in the last inequality we have applied twice \eqref{distance13}.
\end{proof}

Taking the maximum for $z\in \mathbb D$, we arrive at the following result, which we expect to be sharper
than Theorem~\ref{thm_convergence} since in the latter the factor $\kappa(S)^2$ did occur.
\begin{corollary}\label{cor_convergence}
   For the nondegenerate $[m|n]$ Pad\'e approximant $r=p/q$ and the (possibly degenerate)
   $[m-1|n-1]$ Pad\'e approximant $\widetilde r = u/v$ we have that
   $2 (m+n+1)^{{3}/{2}} \,\kappa(\underline{S}) \, \chi_{\mathbb D}(r,\widetilde r)\geq 1$.
\end{corollary}


\subsection{The work of Cabay and Meleshko}\label{sec5.3}
In order to jump over ``numerical blocks'' in the Pad\'e table by some look-ahead procedure, Cabay and Meleshko \cite{CaMe} (see also \cite[Section~3.6]{BGM}) needed to decide whether the $[m|n]$ Pad\'e approximant $r=p/q$ of $f$ is significantly different from the
the $[m-1|n-1]$ Pad\'e approximant $\widetilde r=\widetilde p/\widetilde q$.
Denoting by $\underline c$ the first column of the rectangular matrix $C$ introduced in
\eqref{homogeneous_system}, and by $\underline C$ the square Toeplitz matrix of order $n$ formed by the other columns, we know from \eqref{homogeneous_system} that, with a suitable scalar $\widetilde e$,
\begin{eqnarray*} &&
     \mbox{vec}(q) = q(0) \left[\begin{array}{cc} 1 \\ - \underline C^{-1} \underline c \end{array}\right], \quad
     \underline C \mbox{vec}(\widetilde q) = \left[\begin{array}{cc} 0 \\  \widetilde e
     \end{array}\right], \quad \mbox{and thus} \quad
     \mbox{vec}(\widetilde q) = \widetilde e \underline C^{-1} e_n
\end{eqnarray*}
where $\widetilde q(z) f(z) - \widetilde p(z) = \widetilde e z^{m+n-1} + \mathcal O(z^{m+n})_{z\to 0}$.  The authors in \cite{CaMe} suggested to use the normalization $\| \mbox{vec}(q) \| = \|\mbox{vec}(\widetilde q) \|=1$ and used the indicator
$$
       \kappa_{CM} = \frac{1}{|q(0) \widetilde e |}
$$
as an estimator for $\| \underline C^{-1} \|$, motivated by parts (i),(ii) of the following statement.

\begin{lemma}\label{lem_CM}
   We have that (i) $\kappa_{CM} \geq \| \underline C^{-1} \|/n$, (ii) $\kappa_{CM} \leq \sqrt{n} \, \| \underline C^{-1} \|^2$, (iii) $\sigma_{n}(C)\geq 1/(n \kappa_{CM})$, and
   (iv) $\kappa_{CM} \sim \| \underline S^{-1} e_{m+n} \| \leq \kappa_{BL}$.
\end{lemma}
\begin{proof}
   The Gohberg-Semencul formula \cite[Theorem~3.6.2]{BGM} tells us that
   $q(0) \widetilde e \underline C^{-1} =A_1 A_2-A_3 A_4$, with the four matrices $A_1, A_2, A_3, A_4$ given by the triangular Toeplitz matrices
   {\small \begin{eqnarray*} &&
       \left[\begin{array}{cccc} q_0 & 0 & \cdots & 0 \\q_1 & q_0 & & \vdots \\
       \vdots & & \ddots & 0 \\ q_{n-1} & \cdots & \cdots & q_0
       \end{array}\right], \,
       \left[\begin{array}{cccc} \widetilde q_{n-1} & \cdots  & \widetilde q_1 & \widetilde q_0 \\ 0 & \widetilde q_{n-1} & & \widetilde q_1 \\
       \vdots & & \ddots & \vdots \\ 0 & \cdots & 0 & \widetilde q_{n-1}
       \end{array}\right], \,
       \left[\begin{array}{cccc} 0 & 0 & \cdots & 0 \\
       \widetilde q_0 & 0 & \cdots & 0 \\
       \vdots & & \ddots & 0 \\ \widetilde q_{n-2} & \cdots & \widetilde q_0 & 0
       \end{array}\right], \,
       \left[\begin{array}{cccc} q_{n} & \cdots  & q_2 & q_1 \\ 0 & q_{n} & & q_2 \\
       \vdots & & \ddots & \vdots \\ 0 & \cdots & 0 & q_{n}
       \end{array}\right].
   \end{eqnarray*}}
   Hence $$
   \frac{1}{\kappa_{CM}} \, \| \underline C^{-1} \| \leq \| A_1 \|_F \| A_2 \|_F + \| A_4 \|_F \| A_3 \|_F \leq \sqrt{\| A_1 \|_F^2  + \| A_4 \|_F^2}\sqrt{\| A_1 \|_F^2  + \| A_4 \|_F^2} = n,
   $$
   as claimed in part (i). By the normalization of the denominators we also find that
   $$
       \kappa_{CM} = \frac{1}{|q(0) \widetilde e |}
       = \| \underline C^{-1} e_n \| \, \sqrt{1 + \| \underline C^{-1} \underline c \|^2}
       \leq \| \underline C^{-1} \| \, \sqrt{\| \underline C^{-1} \underline C \|_F^2 + \| \underline C^{-1} \underline c \|^2} \leq \| \underline C^{-1} \|^2 \, \| C \|_F,
   $$
   where $\| C \|_F \leq \sqrt{n}$ by \eqref{scaling}, implying (ii).
   In view of (i), for establishing (iii) it is sufficient to notice that that
   $$
        \sigma_{n}(\underline C) = \min_{x\neq 0} \frac{\| x^* \underline C \|}{\| x \|} \leq \min_{x\neq 0} \frac{\| x^* C \|}{\| x \|} \leq \sigma_{n}(C).
   $$
   A proof of part (iv) is slightly more involved. Notice first that the normalization $\| \mbox{vec}(\widetilde q) \|=1$ of \cite{CaMe} does not lead to coefficient vectors $x(\widetilde r)$ of norm $1$, but $\| \mbox{vec}(\widetilde q)\|\leq \| x(\widetilde r)\| \leq (1+\| T \|)\| \, \mbox{vec}(\widetilde q)\|\leq (1+\sqrt{m+n+2}) \, \| \mbox{vec}(\widetilde q)\|$ by Lemma~\ref{lem_T_C}, and thus $\| x(\widetilde r) \|\sim 1$.
    We have
   $$
    p(z)\widetilde q(z)-\widetilde p(z)q(z) = q(z) (\widetilde q(z) f(z) - \widetilde p(z)) - \widetilde q(z) (q(z) f(z) - p(z))=q(0) \widetilde e z^{m+n-1}
   $$
   since it is a polynomial of degree at most $m+n-1$, and the powers $z^j$ vanish for $j<m+n-1$. This latter identity can be rewritten as $\underline S x(\widetilde r)= - q(0) \widetilde e \, e_{m+n}$, and thus
   $$
          \kappa_{CM} = \frac{\| \underline S^{-1}e_{m+n} \|}{\| x(\widetilde r) \|}\sim \| \underline S^{-1}e_{m+n} \| \leq \kappa_{BL}
   $$
   the last inequality following directly from the definition of $\kappa_{BL}$. This shows part (iv).
\end{proof}

The algorithm presented in \cite{CaMe} tries out all Pad\'e approximants of type $[m-j|n-j]$ for integer $j$ (i.e., on the same diagonal), and accepts to compute the $[m|n]$ Pad\'e approximant if $\kappa_{CM}$ is sufficiently small.
By Lemma~\ref{lem_CM}(iii), this means that in the Cabay-Meleshko algorithm we only compute robust Pad\'e approximants in the sense of {\bf (P2)}, that is, in the sense of Gonnet, G\"uttel and Trefethen \cite{GGT12}.

Finally, as in the proof of Lemma~\ref{lem_convergence} we get from Lemma~\ref{lem_CM}(iv) that $|z|^{m+n-1} \lesssim \kappa_{CM} \chi(r(z),\widetilde r(z))$ for all $|z|\leq 1$. In particular, a sufficiently small $\kappa_{CM}$ implies that $r$ and $\widetilde r$ are indeed significantly different.


\section{Conclusions}\label{sec6}

In this paper we have presented several results on the sensitivity of $[m|n]$ Pad\'e approximants, as well as on the occurrence of spurious poles. Our findings are expressed
in terms of four matrices, namely a rectangular Toeplitz matrix $C$ as in \cite{GGT12}, a rectangular striped Toeplitz matrix $T$, a square triangular Toeplitz matrix $Q$, and a rectangular Sylvester-like matrix $S=QT$,   see \eqref{homogeneous_system},\eqref{homogeneous_system2},\eqref{def_Sylvester}.
These four matrices satisfy
\begin{eqnarray*} &&
   \mbox{$\| C \| \lesssim 1$, $\| T \| \sim 1$, $\| S \| \sim 1$, $\| Q \| \lesssim 1$ due to scaling, see \eqref{scaling} and Lemma~\ref{lem_T_C},}
\\&&
   \mbox{$\| C^\dagger \| \sim \| T^\dagger \|\sim \kappa(T)$ and $\| T^\dagger \| \lesssim \| S^\dagger \|\sim \kappa(S)$, and $\| Q^{-1} \| \lesssim \| S^\dagger\|$,
   see Lemma~\ref{lem_T_C}}.
\end{eqnarray*}
We introduced a kind of hierarchical classification of $[m|n]$ Pad\'e approximants:
there are first the so-called nondegenerate Pad\'e approximants $r=p/q$ considered before in \cite{WW} which can characterize equivalently by one of the following properties
\begin{itemize}
\item the polynomials $p$ and $q$ are co-prime, and that the defect
$\min\{ m- \deg p , n - \deg q \}$ is equal to zero (see {\bf (P1)}), in other words, they correspond to entries located on the left or upper border of a block in the Pad\'e table in exact arithmetic;
\item the Pad\'e map is continuous, see \cite{WW} and Theorem~\ref{thm_continuity};
\item the matrix $S$ and hence $T$ and $C$ have full row rank, see Lemma~\ref{lem_degeneracy}.
\end{itemize}
Secondly there is the subclass of so-called robust Pad\'e approximants in the sense of \cite{GGT12} characterized by a 
sufficiently large $\sigma_{n}(C)$, or, equivalently, a modest $\kappa(T)$. We show that here
\begin{itemize}
   \item the real Pad\'e map is forward well-conditioned, but not necessarily backward, see Theorem~\ref{thm_stability}(c),(d)
       and Example~\ref{exa3};
   \item the Cabay-Meleshko algorithm \cite{CaMe} of \S\ref{sec5.3} computes also robust Pad\'e approximants along a diagonal, it is most of the times cheaper than the approach of \cite{GGT12} since it is recursive, but it might miss some robust approximants since the estimator $\kappa_{CM}$ might be larger than $\kappa(T)$;
\end{itemize}
Finally we have introduced in this paper the class of so-called well-conditioned Pad\'e approximants characterized by a modest $\kappa(S)$, which is hence a subclass of that of robust approximants. For these approximants we have established the following properties
\begin{itemize}
   \item we can control both Froissart doublets, namely the Euclidian distance between poles and zeros of $r$ in the unit disk, as well as small residuals in the disk, see Theorem~\ref{thm_rational};
   \item the real Pad\'e map is backward well-conditioned since $\| Q^{-1} T \|\lesssim \kappa(S)$, see Theorem~\ref{thm_stability}(d);
   \item it is equivalent to measure the distance to $\widetilde r\in \mathcal R_{m,n}$ through the uniform chordal metric in the unit disk or through the difference of normalized coefficient vectors, see Theorem~\ref{thm_distance};
   \item its numerator and denominator are numerically coprime in the sense of \cite{BeLa,CGTW}, see \S\ref{sec5.2}.
\end{itemize}
In the introduction we mentioned the question from \cite{GGT12} whether robust approximants do not have Froissart doublets nor small residuals.
Our Example~\ref{exa3} shows that such a statement is wrong in general, but we
were able to give a positive answer at least for well-conditioned Pad\'e approximants.

We can also draw from Theorem~\ref{thm_rational}, Theorem~\ref{thm_convergence}  and Remark~\ref{rem_residual} the conclusion that it is impossible to find well-conditioned Pad\'e approximants close to $f$ in $\mathbb D$ with small error for
functions  $f$ having themselves small residuals or Froissart doublets in the disk. However, the scaling assumption \eqref{scaling} at least asymptotically scales the complex plane in a way that $f$ will have no singularities in the (open) disk.

More important, Theorem~\ref{thm_convergence} and even more Corollary~\ref{cor_convergence} seem to indicate that there are only finitely many well-conditioned Pad\'e approximants along a fixed diagonal which are close to $f$ in the whole unit disk. 

For future work, it seems for us desirable to get a deeper understanding of the link between $\kappa(S)$ and $\kappa(T)$ (beyond the relation $\kappa(T)\lesssim \kappa(S)$), that is, the link between Pad\'e approximants which are robust and those which are well-conditioned.

Also, it would be nice to know whether the lower bounds of, e.g., Theorem~\ref{thm_rational} are sharp. We feel that the lower bounds should not involve unstructured condition numbers but so-called structured condition numbers, the latter taking into account the particular structure
of our matrix $S$, in the spirit of the discussions in \S\ref{sec5.1} and \S\ref{sec5.2}.
This will be further analyzed in a future work.
\vspace{0.8cm}

\noindent{\bf Acknowledgements.}
The authors gratefully acknowledge valuable discussions with Alexander Aptekarev and Stefan G\"uttel. 
We are also grateful for remarks of the referees which helped us improving the presentation.

\texttt{
Bernhard Beckermann, Ana C. Matos,\\
$\{$bbecker, matos$\}$@math.univ-lille1.fr\\
Laboratoire de Math\'ematiques P. Painlev\'e
UMR CNRS 8524 - Bat.M2\\
Universit\'e Lille - Sciences et Technologies \\
F-59655 Villeneuve d'Ascq Cedex, FRANCE
}

\end{document}